\documentclass[12pt]{article}
\usepackage[authoryear,round]{natbib}
\usepackage{graphicx}

\usepackage{amsthm}
\usepackage{amsmath,amssymb}
\usepackage{latexsym}
\usepackage{adjustbox}
\usepackage{tabularx}
\usepackage{multicol}

\usepackage{placeins}

\newtheorem{theorem}{Theorem}

\newtheorem{lemma}{Lemma}

\makeatletter
\renewenvironment{proof}[1][\proofname]{%
   \par\pushQED{\qed}\normalfont%
   \topsep6\p@\@plus6\p@\relax
   \trivlist\item[\hskip\labelsep\bfseries#1\@addpunct{.}]%
   \ignorespaces
}{%
   \popQED\endtrivlist\@endpefalse
}
\makeatother

\usepackage[utf8]{inputenc}

\usepackage{mathtools}
\mathtoolsset{showonlyrefs,showmanualtags}

\newtheorem{corrolary}{Corrolary}

\providecommand{\keywords}[1]{\textbf{\textit{Index terms---}} #1}

\title{Data-adaptive smoothing for optimal-rate estimation of possibly non-regular parameters}

\author{Aurélien F. Bibaut \and Mark J. van der Laan}

\date{\today}
 
\begin{document}

\title{Data-adaptive smoothing for optimal-rate estimation of possibly non-regular parameters}

\author{Aurélien F. Bibaut \and Mark J. van der Laan}

\maketitle

\begin{abstract}
We consider nonparametric inference of finite dimensional, potentially non-pathwise differentiable target parameters. In a nonparametric model, some examples of such parameters that are always non pathwise differentiable target parameters include probability density functions at a point, or regression functions at a point. In causal inference, under appropriate causal assumptions, mean counterfactual outcomes can be pathwise differentiable or not, depending on the degree at which the positivity assumption holds.

In this paper, given a potentially non-pathwise differentiable target parameter, we introduce a family of approximating parameters, that are pathwise differentiable. This family is indexed by a scalar. In kernel regression or density estimation for instance, a natural choice for such a family is obtained by kernel smoothing and is indexed by the smoothing level. For the counterfactual mean outcome, a possible approximating family is obtained through truncation of the propensity score, and the truncation level then plays the role of the index.

We propose a method to data-adaptively select the index in the family, so as to optimize mean squared error. We prove an asymptotic normality result, which allows us to derive confidence intervals. Under some conditions, our estimator achieves an optimal mean squared error convergence rate. Confidence intervals are data-adaptive and have almost optimal width.

A simulation study demonstrates the practical performance of our estimators for the inference of a causal dose-response curve at a given treatment dose.
\end{abstract}


\section{Introduction}

\subsection{Statistical formulation and estimation problem}

We observe $n$ i.i.d. observations $O_1,...,O_n$ of a random variable $O$ following a probability distribution $P_0$ (the data-generating distribution). We assume that $P_0$ belongs to a set $\mathcal{M}$ of distribution probabilities on the observation space $\mathcal{O} \subset \mathbb{R}^d$, $d \in \mathbb{N}^*$. The set $\mathcal{M}$ is called the statistical model. In this paper, we consider semi-parametric or nonparametric models.

Our goal is to estimate a parameter $\Psi(P_0)$ of the data-generating distribution $P_0$. The functional $\Psi : \mathcal{M} \rightarrow \mathbb{R}$ is called the target parameter mapping.

Let $P_n$ be the empirical probability distribution based on observations $O_1,...,O_n$. In some cases, $P_n$ does not belong to the model $\mathcal{M}$ on which the mapping $\Psi$ is defined. In theses cases, one usually uses an initial estimator $\hat{P}(P_n)$ of $P_0$, which maps an empirical distribution into the smoother model $\mathcal{M}$. When $\Psi$ can be defined directly on the set of empirical distributions, we might just take $\hat{P}$ to be the identity mapping.

\paragraph{Example 1: probability density function at a point.} In this example, we consider the problem of estimating a univariate probability denisity function at a point. In this context, $\mathcal{O} \subset \mathbb{R}$. We will set $\mathcal{M}$ to be the set of probability distributions that admit a density with respect to (w.r.t.) the Lebesgue measure: $\mathcal{M} \equiv \{ P : \exists p \ 	\frac{dP}{d\mu} = p \}$. Given an $x \in \mathcal{O}$, we consider the target parameter $\Psi_{x}(P_0) \equiv p(x) \equiv \frac{dP}{d\mu}(x)$.

\paragraph{Example 2: counterfactual outcome under known treatment mechanism.} In this example, we observe $n$ i.i.d. realizations $O_1,...,O_n$, corresponding to $n$ indivudals. For each individual $i$, $O_i = (W_i, A_i, Y_i)$, $W_i$ represents a set of baseline covariates (e.g. age, sex, biomarker measurements), $A_i$ is binary indicator of whether individual $i$ received a given drug, and $Y_i$ is a binary health outcome (e.g. $Y_i = 1$ patient $i$ is still sick after some time, $Y_i = 0$ if not). 

In this example, our goal will be to estimate $\Psi(P_0) \equiv E_{P_0} E_{P_0} [Y | A = 1, W]$. Under some causal assumptions, this target parameter is equal to the mean counterfactual outcome in the situation in which every patient receives treatment. 

We assume that the probability distributions in our model have densities with respect to an appropriate dominating measure $\mu$ (a product of Lebesgue measures and counting measures): for all $P \in \mathcal{M}$, there exists $p$ such that $p = \frac{dP}{d\mu}$. For  $o$ in $\mathcal{O}$, we have the following factorization: $p(o) = p(y|a,w)p(a|w)p(w)$. For every $P \in \mathcal{M}$, we will denote $q_Y(o) \equiv p(y|a,w)$, $q_W(w) = p(w)$, $q(o) = q_Y(o) q_W(w)$, $g(o) = p(a|w)$. In this example, we will assume that $g_0$ is known.

\paragraph{Example 3: dose-response curve at a fixed dose value, under known treatment mechanism.} We use the same notation for the observed data as in example 1, with the difference that $A_i$ is now continuous and takes values in $[0,1]$. We use the same notations as in example 2 when applicable. We assume that $g_0$ is known here too.

Our target parameter of interest is $\Psi_{a_0}(P_0) \equiv E_{P_0} E_{P_0} \left[Y | A = a_0, W \right]$, where $a_0 \in [0,1]$. Under appropriate causal assumptions, this represents the mean counterfactual outcome in a world in which every patient receives treatment dose $a_0$.

\medskip

\subsection{Pathwise differentiability and efficiency bound}

Pathwise differentiability relative to $\mathcal{M}$ (see \cite{pfanzagl1990} and \cite{BickelKlaasesenRitovWellner1993}) of a target parameter $\Psi$ at $P \in \mathcal{M}$ implies the  first following order expansion (\cite{marks_phdthesis}, \cite{appendix_TLbook}, \cite{first_order_expansion-vdL1995}):
\begin{equation}
\Psi(P) - \Psi(P_0) = -P_0 D^*(P) + R(P, P_0),
\end{equation} where
\begin{itemize}
\item $D^*(P) \in L_2(P)$, and is called the canonical gradient of $\Psi$ at $P$,
\item $R(\cdot, \cdot)$ is a second order term in the sense that, for any parametric submodel $\{P_{\epsilon} : \epsilon \} \subset \mathcal{M}$ such that $P_{\epsilon = 0} = P$, we have that $\epsilon^{-1} R(P_\epsilon, P) \xrightarrow{\epsilon \rightarrow 0} 0$.
\end{itemize}

In many practical situations second order term $R(\cdot, \cdot)$ has the double robustness structure, i.e. $R(P, P_0) = \int (Q(P) - G(P_0)) \times ( G(P) - G(P_0) ) \times H(P, P_0) dP_0$ for some parameters $Q$ and $G$.

Efficiency theory \citep{BickelKlaasesenRitovWellner1993} tells us that the asymptotic variance of any regular estimator of $\Psi(P)$ is at least as large as $Var_P(D^*(P))$.

\paragraph{Example 1, continued.} Under infinite dimensional models, the probability density function (p.d.f) at a point is not pathwise differentiable. 

\paragraph{Example 2, continued.} Under some data-generating distributions the counterfactual mean outcome is not pathwise differentiable either. Even when it is pathwise differentiable, researchers often prefer to target other parameters than $\Psi(P_0)$, as the variance of the canonical gradient of $\Psi$ at $P_0$ can be large if the propensity score $g(a|w)$ is small in some areas of the population.

\paragraph{Example 3, continued.} Under infinite dimensional models, the dose-response curve at a fixed treatment dose is not pathwise differentiable.

\subsection{Smoothed target parameters}

When the target parameter of interest is non-pathwise differentiable, or has large variance of its canonical gradient, one approach to estimation consists in introducing a target "smoothed" version of the target parameter. We will consider a family 
\begin{equation}
\mathcal{F} \equiv \{\Psi_\delta : \mathcal{M} \rightarrow \mathbb{R} \  |\  \delta \in [0, \delta_0], \Psi_{0} = \Psi \}.
\end{equation}

We will assume that for any $\delta > 0$, the target parameter $\Psi_{\delta}$ is pathwise differentiable at any $P \in \mathcal{M}$.

\medskip

Let us present natural smoothing schemes in the context of the three above examples.

\paragraph{Example 1, continued.} The density at a point $x$, under $P$, which we denoted $\Psi_{x}(P)$ can be approximated by $\Psi_{x, \delta}(P) \equiv E_P [\delta^{-1}K((O - x) / \delta)]$, where $K : \mathbb{R} \rightarrow \mathbb{R}$ is a smooth non-negative function, such that $\int K = 1$ and $\int K^2 < \infty$. Observe that $\Psi_{x, \delta}(P_0)$ is the target parameter of the kernel density estimator with kernel $K$ and bandwidth $\delta$. Under some mild smoothness condition on the underlying density, $\Psi_{x, \delta}(P) \xrightarrow{\delta \rightarrow 0} \Psi_x(P)$. The canonical gradient of $\Psi_{x, \delta}$ at $P$ is given by 
$$D^*_{x, \delta}(P) = K_{\delta, x}(O - x) - P K_{\delta, x}(O - x),$$ where we define $K_{\delta, x} \equiv \delta^{-1} K((O - x) / \delta)$. Note that $Var_P D_{x, \delta}(P) \xrightarrow{\delta \rightarrow 0} \infty$

\paragraph{Example 2, continued.} When using estimators such as the Inverse Probability of Treatment Weighted estimator (IPTW estimator, see for instance \cite{robins-hernan-brumback2000}), causal inference practitioners often truncate the propensity scores: in other words, they replace $g(A_i | W_i)$ by $g_{0, \delta}(A_i|W_i) \equiv \max(g(A_i|W_i), \delta)$, for some fixed $\delta > 0$. Truncation has the effect of reducing the variance of the estimators. However it makes the IPTW consistent for another target parameter, $\Psi_{\delta}(P_0) \equiv E_{P_0} [g_0(A|W)/g_{0, \delta} (A|W) \times \bar{Q}(1,W)]$, where $\bar{Q}(a, w) = E_P[Y| A = a, W = w]$. With this definition, we have that $\Psi_{0}(P) = \Psi(P)$. The canonical gradient of $\Psi_{\delta}(P)$ is given by $$D^*_\delta(P) = \frac{A}{g_{0, \delta} (1|W)} (Y - \bar{Q}(A, W)) + \frac{g_0(1|W)}{g_{0, \delta}(1|W)} \bar{Q}(1, W) - \Psi_{\delta}(P).$$ One can readily show that $Var_P\left(D^*_\delta(P)\right)$ increases as $\delta$ decreases, and that under some distributions $P$, it tends to infinity as $\delta$ converges to zero.

\paragraph{Example 3, continued.} One can obtain a pathwise differentiable approximation of $\Psi_{a_0}(P)$ by smoothing. We define $\Psi_{a_0, \delta}(P) \equiv \int_{a} K_{\delta, a_0}(a) \bar{Q}(a, W) da$, where $K_{\delta, a_0}(a) \equiv \delta^{-1} K((a - a_0)/\delta )$, and $K$ is the kernel introduced in example 1 above. 
The canonical gradient of $\Psi_{a_0, \delta}$ is given by
$$D^*_{a_0, x}(P) \equiv \frac{K_{a_0, \delta}(A)}{g_0(A|W)} (Y - \bar{Q}(a, W)) + \int_a K_{a_0, \delta}(a) \bar{Q}(a, W) da - \Psi_{a_0, \delta}(P).$$
We have that $\Psi_{a_0, \delta}(P) \xrightarrow{\delta \rightarrow 0} \Psi_{a_0}(P)$. One can readily show that $Var_P(D^*_\delta(P)) \allowbreak \sim C \delta^{-1}$ for some positive constant $C$, and that bias converges to zero as $\delta$ tends to zero.

\bigskip

Since the all the above smoothed parameters are pathwise differentiable, we have that
\begin{equation}\label{first_order_expansion-smoothed}
\Psi_\delta(P) - \Psi_\delta(P_0) = -P_0 D^*_\delta(P) + R_\delta(P, P_0),
\end{equation}
where $D^*_\delta(P)$ is the canonical gradient of $\Psi_\delta$ at $P$ and $R_\delta(P, P_0)$ is a second order term such that $R_\delta(P, P) = 0$ for all $P$.

The smoothed parameters can be estimated at root-$n$ rate. However, smoothing introduces bias with respect to the parameter one really wants to estimate. Therefore, consistent estimation requires to use a smoothing parameter $\delta_n$ that tends to zero. Ideally, one would want to choose a value of $\delta_n$ that minimizes mean squared error with respect to our target $\Psi(P_0)$. If one uses an asymptotically linear efficient estimator of $\Psi_{\delta_n}(P_0)$, the mean squared error with respect to $\Psi(P_0)$ roughly decomposes as
\begin{equation}
MSE_n(\delta) \approx \frac{1}{n} \sigma_0^2(\delta) + b_0(\delta)^2,
\end{equation}
where $\sigma_0^2(\delta) = Var_{P_0}\left(D^*_{\delta_n}(P_0)\right)$ and $b_0(\delta) = \Psi_{\delta_n}(P_0) - \Psi(P_0)$.

\subsection{Proposed method}

\paragraph{Notations.} First, we will say that two random sequences $(a_n)$ and $(b_n)$ are \textit{asymptotically equivalent in probability}, which we will denote $a_n \sim_P b_n$, if $a_n / b_m \xrightarrow{P} 1$. Secondly, for a function $f$ of a real variable $x$, we will denote $f'(x) \equiv (df/dx)(x)$, whenever this quantity exists.

\medskip

Let us now present our approach. We start out with a class of estimators $\{ \widehat{\Psi}_n(\delta) : \delta \geq 0 \}$ where, for every $\delta \geq 0$, $\widehat{\Psi}_n(\delta)$ is an regular, asymptotically linear efficient, double robust estimator of $\Psi_\delta(P_0)$. For instance, we might take $\widehat{\Psi}_n(\delta)$ to be a one-step estimator \citep{BickelKlaasesenRitovWellner1993} of $\Psi_\delta(P_0)$. Then we propose a data-adaptive selector $\hat{\delta}_n$ of the optimal smoothing level. Finally, we return $\widehat{\Psi}_n(\hat{\delta}_n)$ as our estimate of $\Psi(P_0)$.

Under some mild assumptions, we will prove that
 $\widehat{\Psi}_n(\hat{\delta}_n)$ is asymptotically normally distributed. Under some additional assumptions, we will show that our estimator $\widehat{\Psi}_n(\hat{\delta}_n)$ is optimal in mean squared error rate (w.r.t. $\Psi(P_0)$) among all estimators of the form $\widehat{\Psi}_n(\delta_n)$ where $\delta_n \rightarrow 0$.

\medskip

We now describe the rationale behind our method. It is easiest to understand by first looking at why the most natural methods fail.

The seemingly easiest way to select the smoothing parameter would be to minimize, with respect to $\delta$, an estimate of the mean squared error $MSE_n(\delta)$. However natural estimators of $MSE_n(\delta)$ are hard to find in general. In the case where $\Psi$ is not pathwise differentiable at $P_0$, estimating $MSE_n(\delta)$ has to be at least as hard to estimate as estimating $\Psi(P_0)$ itself. In fact, it decomposes as bias w.r.t. $\Psi(P_0)$, which has to be as hard to estimate as $\Psi(P_0)$ itself, plus a variance term $n^{-1} \sigma^2_0(\delta)$, which is pathwise differentiable.

We examined a tempting fix to the previous approach. As estimation of bias is problematic, we turned to a criterion that involves only estimating $\sigma(\delta)$ and small variations of $\Psi_{\delta}(P_0)$  for small $\delta$ gaps. Unlike $MSE_n(\delta)$, such a criterion would thus be pathwise differentiable, as it would only involve pathwise differentiable quantities. The following observation led us to such a criterion: under some mild smoothness assumptions, the solution to
$n^{-1} \sigma_0'(\delta) + b_0'(\delta) = 0$ converges to zero at the same rate as $\delta^*_{0,n}$. (One way to understand this is that in many problems, for instance in bandwidth optimization in density estimation, at the optimal smoothing level, bias and variance are of same order. Therefore, taking the derivative of $MSE_n(\delta)$ and simplifying leads to such a criterion, up to some constant factors). We thus reckoned that using an estimate of the finite difference approximation
$n^{-1/2}(\sigma_0(\delta + \Delta) - 
\sigma_0(\delta)) / \Delta + (b_0(\delta + \Delta) - b_0(\delta)) / \Delta = 0$, for some appropriately small $\Delta$ should give a good estimate of $\delta^*_{0,n}$.  We expected that estimating its left hand side and finding a root would give us a smoothing parameter that converges to zero at about the same rate as $\delta_{0,n}^*$. However, both simulations and analytic calculations show that for $\delta \lesssim \delta^*_{0,n}$, the standard deviation of the canonical gradient of the criterion exceeds the criterion itself. Therefore, for $\delta \lesssim \delta^*_{0,n}$, the standard deviation of the estimated criterion will be larger than the criterion itself. This is visualized on figure \ref{figure-signal-to-noise-ratio} below.

\begin{figure}[h]
\centering
\includegraphics[width = 8cm, height=6cm]{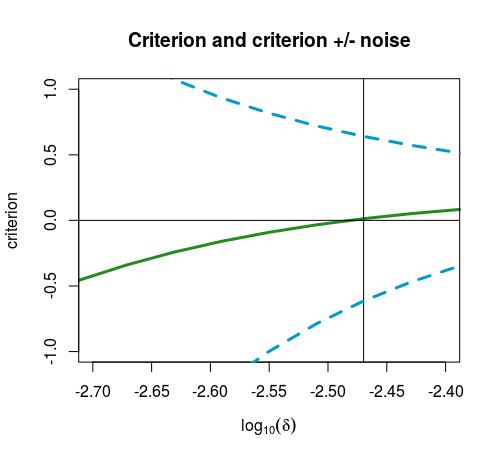}
\caption{The green line represents the criterion $\frac{1}{\sqrt{n}}\sigma'(\delta) + b'(\delta)$. The dotted blue lines represent pointwise confidence bands for the criterion, based on standard deviation of the canonical gradient. This was obtained in the setting of the mean counterfactual outcome problem described in example 2. We used a gap value $\Delta(n, \delta)$ that depends on $\delta$ and $n$, and which was chosen so as to minimize mean squared error of the estimated criterion w.r.t. the true criterion. The vertical black line corresponds to $\delta = \delta^*_{0,n}$. One can observe on this figure that for $\delta \lesssim \delta^*_{0,n}$ the width of the confidence interval largely exceeds the value of the targeted criterion.}
\label{figure-signal-to-noise-ratio}
\end{figure}

\medskip

Our proposed method still aims at solving $MSE'_n(\delta) = 0$, while avoiding the pitfalls we just mentioned. Otherwise stated, we want to estimate, potentially up to a constant, $\delta^*_{0,n}$ that solves $MSE'_n(\delta) = 2 n^{-1} \sigma_0(\delta) \sigma'_0(\delta) + 2 b_0(\delta) b'_0(\delta) = 0$. Our approach relies on several observations.

First, under smoothness assumptions, derivatives can be approximated by finite differences: for small $\Delta$, $b'_0(\delta) \approx \left(\Psi_{\delta + \Delta}(P_0) - \Psi_{\delta}(P_0)\right) / \Delta$, and $\sigma_0'(\delta) \approx  \left(\sigma_0(\delta + \Delta) - \sigma_0(\delta)\right) / \Delta$.

Second, while $b_0(\delta)$ is hard to access, its rate in $\delta$ can be linked to the rate in $\delta$ of $b'_0(\delta)$. Under smoothness assumptions, if $b'_0(\delta) \asymp \delta^{\beta - 1}$ then $b_0(\delta) \asymp \delta^\beta$.

Third, as we mentioned above, at a given $n$, signal to noise ratio for $\sigma_0(\delta)$, $\sigma_0'(\delta)$ and $b_0'(\delta)$ is low for $\delta \lesssim \delta^*_{0,n}$ but high for  $\delta \gg \delta^*_{0, n}$. We can thus perform consistent estimation of $\sigma_0(\tilde{\delta}_n)$, $\sigma'_0(\tilde{\delta}_n)$, $b'_0(\tilde{\delta}_n)$ for a sequence $(\tilde{\delta}_n)$ that converges to zero slower than $\delta_{0,n}^*$.

Finally, under smoothness assumptions, the asymptotic behaviors of $b_0'(\delta)$, $\sigma_0(\delta)$, $\sigma'_0(\delta)$ as $\delta$ converges to zero can be learned by estimating these functions at small values of $\delta$. For instance, if $\sigma_0(\delta) \asymp \delta^{-\gamma}$, then the rate $\gamma$ can be learned by estimating $(\log \sigma(\tilde{\delta}_{1,n}) - \log \sigma_0(\tilde{\delta}_{2,n}) ) / (\log \tilde{\delta}_{1,n} - \log \tilde{\delta}_{2,n})$.

\medskip

These observations lead us to the following method. (We use simplified notations for now to ease exposition). We assume that the asymptotic standard deviation $\sigma_0(\delta)$ of the estimator, the derivative $\sigma'_0(\delta)$ of this asymptotic standard deviation, and the derivative of the asymptotic bias $b_0'(\delta)$ behave as polynomials in $\delta$ as $\delta$ tends to zero: $\sigma_0(\delta) \sim C_\sigma \delta^{-\gamma}$, $\sigma'_0(\delta) \sim C_{\sigma'} \delta^{-\nu}$ and $b_0'(\delta) \sim C_{b'}\delta^{\beta - 1}$. Note that this implies that $b_0(\delta) \sim C_b \delta^{\beta}$. We consider two positive sequences $\tilde{\delta}_{1,n}$ and $\tilde{\delta}_{2,n}$ that converge to zero slowly (this will be made precise later). Estimating $\sigma_0(\tilde{\delta}_{i,n})$, $(\sigma_0(\tilde{\delta}_{i,n} + \Delta_n) - \sigma_0(\tilde{\delta}_{i,n})) / \Delta_n$ and $(\Psi_{\tilde{\delta}_{i,n} + \Delta_n}(P_0) - \Psi_{\tilde{\delta}_{i,n}}(P_0)) / \Delta_n$, for $i = 1,2$ and an appropriate sequence $\Delta_n$, allows us to estimate the powers $\beta$, $\gamma$ and $\nu$, as well as the constants $C_\sigma$, $C_{\sigma'}$ and $C_{b'}$. In other words, by computing estimates along slow sequences $\tilde{\delta}_{i,n}$, $i=1,2$, we learn the asymptotics as $\delta$ converges to zero of $b_0'$, $\sigma_0'$, $\sigma_0$. Using the asymptotic expressions of $\sigma_0$, $\sigma_0'$, $b_0'$, $b_0$, we can express $\delta_{0,n}^*$ from the constants and the powers $\beta$, $\nu$ and $\gamma$. Replacing these by our estimates, we obtain an estimated optimal smoothing rate $\hat{\delta}_n$. We then compute an estimate of the smoothed parameter $\Psi_{\hat{\delta}_n}(P_0)$ using an asymptotically linear efficient, double robust estimator, such as a one-step estimator.

\subsection{State of the art}

\paragraph{Data-adaptive smoothing in density estimation and regression.} There is an abundant literature dealing with adaptive smoothing in nonparametric statistical estimation and prediction. Note that a lot of work in these areas is concerned with estimation of the entire regression function or density function, whereas we address estimation on finite dimensional parameters, such as these density or regression functions at a given point. \cite{stone1984} proposes an asymptotically optimal bandwidth selector in kernel density estimation. This selector has a leave-one-out cross validation interpretation.  \cite{hardle-marron1985} provides a method to select the bandwidth in nonparametric kernel regression, which is asymptotically optimal in mean integrated square error (MISE). \cite{silverman1984} introduces a bandwidth selector that is computationally efficient and asymptotically optimal in MISE. Both of these latter methods rely on leave-one-out cross validation.
\cite{hardle1993} gives a broad review of adaptive smoothing in nonparametric regression. \cite{vdl-dudoit-keles2004} and \cite{vdL-dudoit-vdVaart2006} provide asymptotic optimality guarantees for likelihood-based V-fold cross-validation. Bandwidth selection in nonparametric regression and density estimation are immediate applications.

\paragraph{Confidence intervals for density and regression function.} \cite{bickel-rosenblatt1973} give a result which allows the construction of uniform confidence bands. \cite{hall1992} presents two bootstrap-based methods to construct pointwise confidence intervals for the density function at a point. Key to the two methods is offseting the bias resulting from smoothing. The first one estimates bias explicitely, through a second order derivative estimation. The second one resorts to an undersmoothing scheme, which makes bias vanish relatively to confidence interval witdh. However these methods are not data-adaptive: prior knowledge of the smoothness of the density is assumed. \cite{low1997} gives minimax results for the construction of confidence intervals in nonparametric problems. Thise make clear that constructing adaptive confidence intervals that are valid over large classes of densities is a hard problem in general.
However \cite{gine-nickl2010} detail the construction of data-adaptive uniform confidence bands for a density function. Their findings are consistent with \cite{low1997} in that that they consider special nonparametric classes of densities.

\paragraph{Optimal smoothing in causal inference problems.} 
The need for smoothing non-pathwise differentiable target parameters arise naturally in many causal inference problems, in particular when considering a continous treatment. \cite{diaz2012} propose a super-learning \citep{polley-vdl2010} based approach to estimation of causal dose-response curves. \cite{edwardKennedy2016} recasts the problem of estimation of a dose-response curve as a kernel regression problem. They select the bandwidth data-adaptively and they provide pointwise confidence intervals.

\paragraph{Propensity score truncation in causal inference.} As explained above, large asymptotic variance or even non-pathwise differentiabily can arise in causal inference when propensity scores take small values. A common approach consists in truncating these propensity scores \citep{petersen2011}. \cite{vdL-bembom2008} proposes a method to data-adaptively select the truncation level in the case where the causal target parameter of interest is pathwise differentiable.

\paragraph{Exceptional laws in optimal dynamic treatments.} The mean counterfactual outcome under an optimal treatment rule is in general non-pathwise differentiable when there is a stratum of the population in which the treatment is neither beneficial nor harmful. \cite{luedtke-vdL2016} manage to provide root-$n$ rate inference in this situation.

\subsection{Contributions and article organization}

In this paper we provide a generally applicable method to (optimally) select the indexing parameter of an approximating family, as presented in section 1.3. An asymptotic normality result, construction of confidence intervals, and an asymptotic optimality result are given under some general (i.e. non problem-specific) conditions. To the best of our knowledge, no such generality is claimed in existing works. 

We check that, in our three aforementioned examples, use of some widely available estimators make our conditions hold. We illustrate the practical performance of our method in the dose-response curve example.

The remainder of this article is organized as follows. In section 2, we introduce the key notations and we define estimators of some of quantities introcuded. We then use these estimators to define an estimator of the target parameter of interest. In section 3, we give the theoretical guarantees of our method. In section 4, we check the assumptions for our three aforementioned examples. In section 5, we report simulations results in the case of the dose-response curve example. Section 6 discusses the method, practical as well as theoretical potential improvements. Most of the proofs are deferred to the appendix.

\section{Estimator}

\subsection{Sample splitting and notations}

We split our sample into three subsamples  $S_{1,n} \equiv \{O_i : i = 1,..., l_{1, n} \}$, $S_{2,n} \equiv \{O_i : i = l_{1, n} + 1,..., l_{2, n} \}$, and $S_{3,n} \equiv\{O_i : i = l_{2, n} + 1,..., n \}$, for some $l_{1, n} \equiv p_1 n$ and $l_{2,n} \equiv p_2 n$ for $0 < p_1 < p_2 < 1$. We will denote $l_{3,n} \equiv n$.

In estimating the optimal smoothing level, we will use only the first two subsamples $S_{1,n}$ and $S_{2,n}$. As we will explain below, we use cross-validated one-step estimators \citep{BickelKlaasesenRitovWellner1993} to this end. $S_{1,n}$ is used to compute an initial estimate $\widehat{P}_{1,n}$ of the likelihood, at which we evaluate canonical gradients. We then average these canonical gradients under $P_{2,n}$, the empirical distribution defined by subsample $S_{2,n}$.

Next, we compute a one-step estimate of the smoothed parameter indexed by the estimated optimal smoothing level. We use $S_{1,n} \cup S_{2,n}$ to give an estimate $\widehat{P}_{2,n}$ of the $P_0$, at which we evaluate the appropriate canonical gradient. We then average this latter under $P_{3,n}$, the empirical distribution defined by subsample $S_{3,n}$.

\medskip

For $i \in \{1,2\}$, let $\sigma^2_{i, n}(\delta) \equiv P_0 (D^*_\delta (\widehat{P}_{i,n}) - P_0 D^*_\delta (\widehat{P}_{i,n}) )^2$.

We denote $D^*_{\delta, \infty}$ the limit in $L_2(P_0)$-norm, as $n$ converges to $\infty$, if it exists, of $D^*_\delta ( \widehat{P}_n )$.
Let $\sigma^2_\infty(\delta) \equiv P_0 (D^*_{\delta, \infty} - P_0 D^*_{\delta, \infty} )^2$.

We define $H_{0,\delta}(P) \equiv (D^*_\delta(P) - P_0 D^*_\delta(P) )^2$ and $H_{0, \delta, \infty} \equiv (D^*_{\delta, \infty} - P_0 D^*_{\delta, \infty})^2$, which will be relevant in the asymptotic analysis of our estimator. 

\medskip

Let us now introduce our key smoothness assumption and the pertaining notations.

\medskip

\textbf{A1}. There exist $C_{b',0}$, $C_{\sigma,\infty}$, $C_{\sigma',\infty}$, $C_H$, $\beta_0 \geq 0$, $\gamma_{0, \infty} \geq 0$, $\nu_{0, \infty} \geq 0$, $\eta_{0, \infty} > 0$ such that 
\begin{gather}
\sigma_\infty(\delta) \sim C_{\sigma,\infty} \delta^{-\gamma_{0, \infty}}, \qquad \qquad
\sigma_\infty'(\delta) \sim C_{\sigma',\infty} \delta^{\nu_{0, \infty}},\\
\qquad
b_0'(\delta) \sim C_{b',0} \delta^{\beta_0}, \qquad \text{and} \qquad P_0 \left(H_{0,\delta, \infty} - P_0 H_{0,\delta, \infty} \right)^2 \sim C_H \delta^{-\eta_{0, \infty}}.
\end{gather}

Furthermore, there exist $k_0 > 0$ and $k_1 > 0$ such that $b''_0(\delta) = O_P \left(\delta^{-k_0}\right)$ and $\sigma''_\infty(\delta) = O_P \left(\delta^{-k_1}\right)$.

\subsection{Estimator definition}

We now define estimators of the rates in $\delta$ of $b_0'(\delta)$, $\sigma_\infty(\delta)$ and $\sigma'_\infty(\delta)$. These will rely on estimators of $\sigma_\infty(\delta)$, $\sigma_\infty'(\delta)$ and $b_0(\delta)$ along slowly vanishing sequences $\tilde{\delta}_{1,n}$ and $\tilde{\delta}_{2,n}$. 

For the sake of rate estimation, we use the cross-validated one-step estimator
\begin{equation}
\widehat{\Psi}_{2,n}(\delta) \equiv \Psi_{\delta}(\widehat{P}_{1,n}) + P_{2,n} D^*_{\delta}(\widehat{P}_{1,n})
\end{equation}

to estimate $\Psi_{\delta}(P_0)$, 
\begin{equation}
\widehat{b'}_{2,n}(\delta, \Delta) \equiv \frac{\widehat{\Psi}_{2,n}(\delta + \Delta) - \widehat{\Psi}_{2,n}(\delta)}{\Delta}
\end{equation}

as an estimator of $\left(b_0(\delta + \Delta)- b_0(\delta)\right) / \Delta$,
\begin{equation}
\widehat{\sigma}^2_{2,n}(\delta) \equiv P_{2,n} (D^*_{\delta} (\widehat{P}_{1,n}) - P_{2,n} D^*_{\delta}(\widehat{P}_{1,n}) )^2
\end{equation}

as an estimator of $\sigma^2_\infty(\delta)$, and
\begin{equation}
\widehat{\sigma'}_{2,n}(\delta, \Delta) \equiv \frac{\widehat{\sigma}_{2,n}(\delta + \Delta) - \widehat{\sigma}_{2,n}(\delta)}{\Delta}
\end{equation}

as an estimator of $\left(\sigma_\infty(\delta + \Delta) - \sigma_\infty(\delta) \right) / \Delta$.

\medskip

Let $\Delta_{n} \equiv \left(l_{2,n} - l_{1,n}\right)^{-1/4}$. We take the aforementionned sequences $\tilde{\delta}_{1,n}$ and $\tilde{\delta}_{2,n}$ to be vanishing sequences that go to zero at a slow enough rate, which will be made precise later. We estimate the rates $\beta_0$, $\gamma_{0, \infty}$ and $\nu_{0, \infty}$, and the constants $C_{b',0}$, $C_{\sigma,\infty}$, $C_{\sigma',\infty}$ using respectively
\begin{gather}
\widehat{\beta}_n \equiv \frac{\log \widehat{b'}_{2,n}(\tilde{\delta}_{2,n}, \Delta_n) - \log \widehat{b'}_{2,n}(\tilde{\delta}_{1,n}, \Delta_n)}{\log \tilde{\delta}_{2,n} - \log \tilde{\delta}_{1,n}},\label{beta_hat_definition}\\
\widehat{\gamma}_n \equiv \frac{\log \widehat{\sigma}_{2,n}(\tilde{\delta}_{2,n}) - \log \widehat{\sigma}_{1,n}(\tilde{\delta}_{1,n})}{\log \tilde{\delta}_{2,n} - \log \tilde{\delta}_{1,n}},\label{gamma_hat_definition}\\
\widehat{\nu}_n \equiv \frac{\log \widehat{\sigma'}_{2,n}(\tilde{\delta}_{2,n}, \Delta_n) - \log \widehat{\sigma'}_{1,n}(\tilde{\delta}_{1,n}, \Delta_n)}{\log \tilde{\delta}_{2,n} - \log \tilde{\delta}_{1,n}}\label{nu_hat_definition},
\end{gather}
and
\begin{gather}
\widehat{C}_{b',n} \equiv \widehat{b'}_{2,n}(\tilde{\delta}_{3,n}, \Delta_n) \tilde{\delta}_{3,n}^{-\hat{\beta}_n},\\
\widehat{C}_{\sigma, n} \equiv \widehat{\sigma}_{2,n}(\tilde{\delta}_{3,n}) \tilde{\delta}_{3,n}^{\hat{\gamma}_n},\\
\widehat{C}_{\sigma', n} \equiv \widehat{\sigma'}_{2,n}(\tilde{\delta}_{3,n}) \tilde{\delta}_{3,n}^{-\hat{\nu}_n}.
\end{gather}

We now turn to the estimation of the optimal smoothing rate. Following the arguments made earlier, we should have $MSE'_n(\delta) \approx 2 (l_{3,n} - l_{2,n})^{-1} \times\sigma_0(\delta) \sigma'_0(\delta) + 2 b_0(\delta) b'_0(\delta)$ (as our cross-validated estimator uses only one split, we expect its variance to scale as $ (l_{3,n} - l_{2,n})^{-1}$ instead of $n^{-1}$). Under \textbf{A1}, it is thus natural to expect that, asymptotically, $MSE_n'(\delta_{0,n}^*) \sim C_{b',0}^2 \beta_0^{-1} {\delta_{0,n}^*}^{2\beta_0 - 1} + (l_{3,n} - l_{2,n})^{-1} C_{\sigma', \infty} C_{\sigma, \infty} {\delta_{0,n}^*}^{-\gamma_{0, \infty} + \nu_{0, \infty}}$. (We will show that, under some additional assumptions, this indeed holds.) This would entail that $\delta_{0,n}^* \sim (C_{\sigma, \infty}, C_{\sigma', \infty} \beta_0 C_{b',0}^{-2})^{r_0} \left(l_{3,n} - l_{2,n}\right)^{-r_0}$, with $r_{0, \infty} \equiv (2 \beta_0 - 1 + \gamma_{0, \infty} - \nu_{0, \infty})^{-1}$. This motivates the estimator
\begin{equation}
\widehat{r}_n \equiv \frac{1}{2 \widehat{\beta}_n - 1 + \widehat{\gamma}_n - \widehat{\nu}_n}
\end{equation}

to estimate the optimal smoothing rate and 
\begin{equation}
\widehat{C}_n \equiv \left(\frac{\widehat{C}_{\sigma, n} \widehat{C}_{\sigma', n} \widehat{\beta}_n }{\widehat{C}_{b',n}}\right)^{\widehat{r}_n}
\end{equation}

as an estimator of the constant in the optimal smoothing level $\delta_{0,n}^*$.

\medskip

We finally present our estimator of $\Psi(P_0)$. For a small $\epsilon > 0$, let us define $ \hat{\delta}_{\epsilon, n} \equiv \widehat{C}_n \left(l_{3,n} - l_{2,n}\right)^{-\widehat{r}_n - \epsilon}$. As we will see in the next section, under appropriate assumptions, this sequence is asymptotically slightly faster than the optimal smoothing level $\delta_{0,n}^*$. We define our estimator of $\Psi(P_0)$ as the cross-validated one-step estimator
\begin{equation}
\widehat{\Psi}_n(\hat{\delta}_{\epsilon, n}) \equiv \Psi(\widehat{P}_{2,n}) + P_{3,n} D^*_{\widehat{\delta}_{\epsilon, n}}(\widehat{P}_{2,n}).
\end{equation}

We use a slightly faster-than-optimal smoothing rate $\widehat{r}_n + \epsilon$ in order to make bias vanish relatively to standard error. As it is possible to estimate standard error, this scheme enables the construction of confidence intervals.

\section{Asymptotic analysis}

\subsection{Asymptotic analysis of the smoothing parameter selector}

Consistency of the rate estimators \eqref{beta_hat_definition}, \eqref{gamma_hat_definition}, and \eqref{nu_hat_definition} requires mild additional assumptions that we present here. First, we need that, if we take $\tilde{\delta}_n$ that converges to zero slowly enough, then $\sigma_{1,n}(\tilde{\delta}_n)$, $\sigma'_{1,n}(\tilde{\delta}_n)$, $\sigma''_{1,n}(\tilde{\delta}_n)$, and $P_0 (H_{0,\tilde{\delta}_n} (\widehat{P}_{1,n} ) - P_0 H_{0,\tilde{\delta}_n} (\widehat{P}_{1,n} ) )^2$ are asymptotically equivalent to the limit quantities (where $\widehat{P}_{1,n}$ is replaced by $P_\infty$) $\sigma_{0, \infty}(\tilde{\delta}_n)$, $\sigma'_{0, \infty}(\tilde{\delta}_n)$, $\sigma''_{0, \infty}(\tilde{\delta}_n)$ and $P_0 (H_{0, \infty, \tilde{\delta}_n} - P_0 H_{0, \infty, \tilde{\delta}_n} )^2$. We formalize this in assumptions \textbf{A2} and \textbf{A4} below. Secondly, we need that the remainder term $R_{
\tilde{\delta}_n}(\widehat{P}_{1,n}, P_0)$  remains second-order in the expansion $\Psi_{
\tilde{\delta}_n}(\widehat{P}_{1,n}) - \Psi_{
\tilde{\delta}_n}(P_0) = -P_0 D^*_{
\tilde{\delta}_n}(\widehat{P}_{1,n}) + R_{
\tilde{\delta}_n}(\widehat{P}_{1,n}, P_0)$, provided $\tilde{\delta}_n$ converges to zero slowly enough. We formalize this in assumption \textbf{A3} below. 
Thirdly, we need that the derivative of the remainder term $R_\delta(\widehat{P}_{1,n}, P_0)$, evaluated $\tilde{\delta}_n$, to be bounded by a rate of the form $\tilde{\delta}_n^{-k_1} (l_{2,n} - l_{1,n})^{-\kappa_1}$, provided $\tilde{\delta}_n$ converges to zero slowly enough. We formalize this in assumption \textbf{A5} below.

\medskip

\textbf{A2}. There exist $r^+ > 0$, $k_2 > 0$, $k_2' > 0$, $\kappa_2 > 0$ and $\kappa_2' > 0$ such that for any positive sequence $\tilde{\delta}_n$ that converges to zero slower than \begin{gather}
\sigma_{1,n}(\tilde{\delta}_n) - \sigma_\infty(\tilde{\delta}_n) = O_P(\tilde{\delta}_n^{-k_2} (l_{2,n} - l_{1,n})^{-\kappa_2}), \\
\sigma'_{1,n}(\tilde{\delta}_n) - \sigma'_\infty(\tilde{\delta}_n) = O_P(\tilde{\delta}_n^{-k'_2} (l_{2,n} - l_{1,n})^{-\kappa'_2}),\\
\text{and } \sigma_{1,n}''(\tilde{\delta}_n) \sim \sigma_\infty''(\tilde{\delta}_n).
\end{gather}

\textbf{A3}. There exists $r^+ > 0$ such that for any sequence $\tilde{\delta}_n$ that converges to zero slower than $n^{-r^+}$, $R_{
\tilde{\delta}_n}(\widehat{P}_{1,n}, P_0) = o_P(\sigma_\infty(\tilde{\delta}_n))$.

\textbf{A4}. There exists $r^+ > 0$ such that for any sequence $\tilde{\delta}_n$ that converges to zero slower than $n^{-r^+}$, \begin{gather}
P_0(H_{0,\tilde{\delta}_n} (\widehat{P}_{1,n} ) - P_0 H_{0,\tilde{\delta}_n} (\widehat{P}_{1,n} ) )^2 \sim_P P_0 (H_{0, \infty, \tilde{\delta}_n} - P_0 H_{0, \infty, \tilde{\delta}_n} )^2.
\end{gather}

\textbf{A5}. There exist $k_3 >  0$, $\kappa_3 > 0$ and $r^+ > 0$ such that for any sequence $\tilde{\delta}_n$ that converges to zero slower than $n^{-r^+}$, 
\begin{equation}
\frac{\partial R_{\delta}(\hat{P}_{1,n}, P_0 )}{\partial \delta} \bigg|_{\delta = \tilde{\delta}_n} = O_P (\tilde{\delta}_n^{-k_3} \left(l_{2,n} - l_{1,n} )^{-\kappa_3}\right).
\end{equation}

Note that assumption \textbf{A3} does not necessarily require consistency of $\widehat{P}_{1,n}$. This can be understood by considering for instance second order terms that have the double robustness structure (see e.g. \cite{vdLRobins2003}).

\medskip

We now present our consistency results for the rates estimators. The proofs are deferred to the appendix.
\begin{lemma}\label{lemma-rate-estimators}
Assume \textbf{A1} through \textbf{A5}. Then
\begin{equation}
\widehat{\beta}_n - \beta_0 = o_P \left(\frac{1}{\log n}\right), \qquad \widehat{\gamma}_n - \gamma_{0, \infty} = o_P\left(\frac{1}{\log n}\right), \qquad \widehat{\nu}_n - \nu_{0, \infty} = o_P\left(\frac{1}{\log n}\right),
\end{equation}
and 
\begin{equation}
\widehat{C}_{b',n} \xrightarrow{P} C_{b',0}, \qquad \widehat{C}_{\sigma,n} \xrightarrow{P} C_{\sigma, \infty}, \qquad \widehat{C}_{\sigma', n} \xrightarrow{P} C_{\sigma', \infty}.	
\end{equation}
\end{lemma}

Consistency of the optimal smoothing parameter's rate and constant estimator is then an immediate corrolary, which we now state.

\begin{corrolary}
Assume \textbf{A1} through \textbf{A5}. Then
\begin{equation}
\widehat{r}_n - r_{0, \infty} = o_P\left(\frac{1}{\log n}\right) \qquad \text{and} \qquad \widehat{C}_n \xrightarrow{P} C_0,
\end{equation}
with $C_0 \equiv \left(\frac{C_{\sigma, \infty}, C_{\sigma', \infty} \beta_0}{C_{b',0}^2}\right)^{r_{0, \infty}}$.
\end{corrolary}

\subsection{Asymptotic normality of our estimator}

Asymptotic normality of our estimator $\widehat{\Psi}_{\epsilon,n}$ necessitates the following strengthening of assumptions \textbf{A2} and \textbf{A3}.

\medskip

\textbf{A6}. For $r$ in a neighborhood of the optimal rate $r_{0, \infty}$, $\sigma_{2,n}((l_{3,n} - l_{2,n})^{-r}) \sim_P \sigma_\infty((l_{3,n} - l_{2,n})^{-r})$.

\medskip

\textbf{A7}. For $r$ in a neighborhood of the optimal rate $r_{0, \infty}$, $R_{(l_{3,n} - l_{2,n})^{-r}}(\hat{P}_{2,n}, P_0) = o_P((l_{3,n} - l_{2,n})^{-1/2}\sigma_\infty((l_{3,n} - l_{2,n})^{-r}) + b_0((l_{3,n} - l_{2,n})^{-r}))$.

\medskip

We now state our asymptotic normality result.

\begin{theorem}\label{asymptotic_normality_thm}
Assume \textbf{A1} through \textbf{A7}. Then
\begin{equation}
\widehat{C}_{\sigma, n}^{-1} \hat{\delta}_{\epsilon, n}^{\widehat{\gamma}_n} \left(l_{3,n} - l_{2,n}\right)^\frac{1}{2} \left(\widehat{\Psi}_n(\hat{\delta}_{\epsilon, n}) - \Psi(P_0) \right) \xrightarrow{d} \mathcal{N}(0,1).\label{asymptotic_normality_equation}
\end{equation}
\end{theorem}

\paragraph{Confidence intervals.} Theorem \ref{asymptotic_normality_thm} enables the construction of confidence intervals for $\Psi(P_0)$. Set a confidence level $1 -\alpha$, where $\alpha \in (0,1)$. Let $q_{1 - \alpha / 2}$ be the $(1 - \alpha / 2)$-quantile of the standard normal distribution. Then, if the assumptions \textbf{A1} through \textbf{A4} hold, \eqref{asymptotic_normality_equation} implies that the probability of the event
\begin{equation}
\Psi(P_0) \in CI_{\alpha, \epsilon, n} \equiv  \left[\widehat{\Psi}_n(\hat{\delta}_{\epsilon, n}) \mp q_{1-\alpha/2} \frac{\widehat{C}_{\sigma, n}}{\left(l_{3,n} - l_{2,n} \right)^{\frac{1}{2} - \left(\hat{r}_n + \epsilon\right) \hat{\gamma}_n}} \right] \label{confidence_interval}
\end{equation}
converges to $1 - \alpha$ as $n$ tends to infinity.

\medskip

\begin{proof}[Proof of theorem \ref{asymptotic_normality_thm}]
We have that
\begin{gather}
\left(\widehat{\Psi}_n - \Psi(P_0) \right) = \widehat{\Psi}_n(\widehat{\delta}_{\epsilon, n}) - \Psi_{\widehat{\delta}_{\epsilon, n}}\left(P_0\right) + b_0(\widehat{\delta}_{\epsilon, n}) \\
= \left(P_{3,n} - P_0\right) D^*_{\widehat{\delta}_{\epsilon, n}}(\widehat{P}_{2,n}) + b_0(\widehat{\delta}_{\epsilon, n}) + R_{\widehat{\delta}_{\epsilon, n}}(\widehat{P}_{1,n}, P_0). \label{equation-proof_asymptotic_normality}
\end{gather}

From lemma \ref{lemma-rate-estimators}, $\widehat{\delta}_{\epsilon,n}^{\widehat{\gamma}_n} \widehat{C}^{-1}_{\sigma, n} \sim_P \sigma_\infty(\widehat{\delta}_{\epsilon, n})^{-1}$. Under \textbf{A6}, for $\epsilon$ small enough, $\sigma_\infty(\widehat{\delta}_{\epsilon, n})^{-1} \sim_P \sigma_{2,n}(\widehat{\delta}_{\epsilon, n})^{-1}$.

\medskip

Therefore, reasoning as in the proofs of lemma \ref{lemma-diff_sigma_hat-sigma_inf}, that is, conditionning on $S_{2,n}$, using Lindeberg theorem for triangular arrays, and then applying dominated convergence, we have that
\begin{gather}
\widehat{C}_{\sigma,n}^{-1} \widehat{\delta}_{\epsilon, n}^{\widehat{\gamma}_n} \left(l_{3,n} - l_{2,n} \right)^{\frac{1}{2}} \left(P_{3,n} - P_0\right) D^*_{\widehat{\delta}_{\epsilon, n}}\left(P_{2,n}\right) \xrightarrow{d} \mathcal{N}(0,1). \label{empirical_process-proof_asymptotic_normality}
\end{gather}

Let us now turn to the bias term. We show that the undersmoothing implied by the slightly faster-than-optimal rate $\widehat{\delta}_{\epsilon, n}$ makes the bias term negligible in front of the empirical process term.

\medskip

$\bullet$ From lemma \ref{link_gamma_nu-lemma}, if $\gamma_{0, \infty} \neq 0$, then $\nu_{0, \infty} = -\gamma_{0, \infty} - 1$ and $r_{0, \infty} = \frac{1}{2(\beta_0 + \gamma_{0, \infty})}$.

Then, using lemma \ref{lemma-rate-estimators}
\begin{gather}
\widehat{C}_{\sigma, n}^{-1} \left(l_{3,n} - l_{2,n} \right)^\frac{1}{2} \widehat{\delta}_{\epsilon, n}^{\widehat{\gamma}_n} b_0(\widehat{\delta}_{\epsilon, n})\\
 \sim_P C_{\sigma, \infty} \left(l_{3,n} - l_{2,n}\right)^{\frac{1}{2} - (r_{0, \infty} + \epsilon)(\gamma_0 + \beta_0) + o_P\left(\frac{1}{\log n}\right)}\\
\sim_P C_{\sigma, \infty}^{-1} \left(l_{3,n} - l_{2,n} \right)^{-\epsilon (\beta_0 + \gamma_0)} = o_P(1).\label{bias_term_case1-proof_asymptotic_normality}
\end{gather}

$\bullet$ From lemma \ref{link_gamma_nu-lemma}, if $\gamma_{0,\infty} = 0$, then $r_{0, \infty} = \frac{1}{2 \beta_0 - 1 - \nu_{0, \infty}}$, $r_{0, \infty} > 0$, and $\nu_{0, \infty} > 0$.

Then, using lemma \ref{lemma-rate-estimators},
\begin{gather}
\widehat{C}_{\sigma, n}^{-1} \left(l_{3,n} - l_{2,n} \right)^\frac{1}{2} \widehat{\delta}_{\epsilon, n}^{\widehat{\gamma}_n} b_0(\widehat{\delta}_{\epsilon, n})\\
 \sim_P C_{\sigma, \infty} \left(l_{3,n} - l_{2,n}\right)^{\frac{1}{2} - (r_{0, \infty} + \epsilon)(\gamma_0 + \beta_0) + o_P\left(\frac{1}{\log n}\right)}\\
\sim_P C^{-1}_{\sigma, \infty} \left(l_{3,n} - l_{2,n} \right)^{-\epsilon (\beta_0 + \gamma_{0, \infty}) - (1+\nu_{0, \infty}) r} = o_P(1).\label{bias_term_case2-proof_asymptotic_normality}
\end{gather}

Finally, we address the remainder term. Take $\epsilon$ to be small enough so that $r + \epsilon$ is in the neighborhood of $r_{0, \infty}$ from assumption \textbf{A7}. From lemma \ref{lemma-rate-estimators}, we have that
\begin{gather}
\left(l_{3,n} - l_{2,n} \right)^\frac{1}{2} \widehat{\delta}_{\epsilon, n}^{\widehat{\gamma}_n} R_{\widehat{\delta}_{\epsilon, n}}(\widehat{P}_{1,n}, P_0) \sim_P \frac{R_{\widehat{\delta}_{\epsilon,n}} (\widehat{P}_{1,n}, P_0)}{\left(l_{3,n} - l_{2,n}\right)^{-\frac{1}{2}} \sigma_\infty(\widehat{\delta}_{\epsilon, n})}\\
= O_P \left(\frac{R_{\widehat{\delta}_{\epsilon,n}} (\widehat{P}_{1,n}, P_0 )}{\left(l_{3,n} - l_{2,n}\right)^{-\frac{1}{2}} \sigma_\infty(\widehat{\delta}_{\epsilon, n}) + b_0(\widehat{\delta}_{\epsilon, n})}\right)\\
= o_P(1). \label{remainder-proof_asymptotic_linearity}
\end{gather}
The second line above is obtained using that $b_0(\widehat{\delta}_{\epsilon, n}) = o_P ((l_{3,n} - l_{2,n})^{-1/2} \times \sigma_\infty(\widehat{\delta}_{\epsilon, n}))$ (which is just the result of the bias term analysis reformulated). The third line results from assumption \textbf{A7}.

Gathering \eqref{equation-proof_asymptotic_normality}, \eqref{empirical_process-proof_asymptotic_normality}, \eqref{bias_term_case1-proof_asymptotic_normality}, \eqref{bias_term_case2-proof_asymptotic_normality}, and \eqref{remainder-proof_asymptotic_linearity} yields \eqref{asymptotic_normality_equation}.
\end{proof}

\paragraph{Case of $\Psi$ pathwise differentiable at $P_0$.} When $\Psi$ is pathwise differentiable at $P_0$, the $Var_{P_0} D^*_\delta(P_0)$ does not tend to infinity as $\delta$ tends to zero. If $P_\infty = P_0$, we have that $\sigma_{0, \infty}^2(\delta) =Var_{P_0} D^*_\delta(P_0)$ and therefore we have that $\gamma_{0, \infty} = 0$. By lemma \ref{lemma-rate-estimators}, we thus have $\hat{\gamma}_n = o_P(1 / \log n )$, which implies that the factor $ \hat{\delta}_{\epsilon, n}^{\widehat{\gamma}_n} (l_{3,n} - l_{2,n})^{1/2}$ from the asymptotic normality equation \eqref{asymptotic_normality_equation} is asymptotically equivalent to $\sqrt{l_{3,n} - l_{2,n}}$. Therefore, in the case where $\Psi$ pathwise differentiable at $P_0$ and $P_\infty = P_0$, our estimator minus its target converges to a normal distribution at root-$n$ rate.

\subsection{Asymptotic optimality in mean squared error}

Let $MSE_n(\delta)$ be the mean squared error of our estimator $\widehat{\Psi}_n(\delta)$ with respect to $\Psi(P_0)$. Formally, $MSE_n(\delta) \equiv E_{P_0}[ (\widehat{\Psi}_n(\delta) - \Psi(P_0))^2]$.

One last assumption is needed for our analysis of the mean squared error of our estimator.

\medskip

\textbf{A8}. For $r$ in a neighborhood of the optimal rate $r_{0, \infty}$, 
\begin{gather}
\frac{d R_{\delta}\left(\widehat{P}_{2,n}, P_0 \right)}{d\delta}\bigg|_{\delta = (l_{3,n} - l_{2,n})^{-r}}\\ = o_P \left(\frac{1}{\sqrt{l_{3,n} - l_{2,n}}}  \sigma_\infty'((l_{3,n} - l_{2,n})^{-r}) + b_0'((l_{3,n} - l_{2,n})^{-r}) \right).
\end{gather}

\begin{theorem}\label{MSE_optimality_theorem}
Assume \textbf{A1} through \textbf{A8}. Then there exists an $\epsilon_0 > 0$ such that for any $0 < \epsilon \leq \epsilon_0$,
\begin{equation}
\frac{MSE_n(\widehat{\delta}_{\epsilon, n})}{MSE_n(\delta^*_{0,n})} \left(l_{3,n} - l_{2,n} \right)^{-2 \gamma \epsilon} \xrightarrow{P} K(p_2),
\end{equation}
where $K(p_2)$ is a constant, which is decreasing in $p_2$.

\end{theorem}

This theorem tells us that for small $\epsilon > 0$, $\widehat{\Psi}_n(\widehat{\delta}_{\epsilon, n})$ achieves an almost optimal mean squared error rate. For $\epsilon = 0$, $\widehat{\Psi}_n(\widehat{\delta}_{0, n})$ has an asymptotically optimal mean squared error rate.

\paragraph{Alternative construction of confidence intervals.} Based on theorem \ref{asymptotic_normality_thm} and \ref{MSE_optimality_theorem}, we propose alternative confidence intervals centered at $\hat{\Psi}_n(\hat{\delta}_{0,n})$ with width scaling as $\sigma_{0, \infty}(\hat{\delta}_{\epsilon, n}) (l_{3,n} - l_{2,n} )^{-1/2}$, for some $\epsilon > 0$:

\begin{gather}\label{alternative-confidence_interval}
CI'_{\alpha, \epsilon, n} \equiv \left[\widehat{\Psi}_n\left(\hat{\delta}_{0,n}\right) \mp q_{1-\alpha/2} \frac{\widehat{C}_{\sigma, n}}{\left(l_{3,n} - l_{2,n} \right)^{\frac{1}{2} - \left(\hat{r}_n + \epsilon\right) \hat{\gamma}_n}} \right].
\end{gather}

One can readily observe that coverage of such confidence intervals converges to one as sample size converges to infinity. For given $\alpha$ and $\epsilon$, this confidence interval presents the advantage over the previously introduced $CI_{\alpha, \epsilon, n}$ that it is centered around a more efficient estimator, while having 
same width.

\section{CV-TMLE version of our estimator}\label{CV-TMLE-section}

Under some stronger assumptions, we can prove the asymptotic normality of a CV-TMLE version of our estimator of $\Psi(P_0)$. Let us define this estimator here.

Let $B_n$ denote a random vector indicating a split of the indices $\{1,..., n\}$ into a training sample $\mathcal{T}_n$ and a validation sample $\mathcal{V}_n$ : $\mathcal{T}_n \equiv \{i : B_n(i) = 0 \}$ and $\mathcal{V}_n \equiv \{i : B_n(i) = 1 \}$. We denote $P^0_{n, B_n}$ the empirical distribution on the training sample and $P^1_{n,B_n}$ the empirical distribution on the validation sample. Let $\widehat{P}^0_{n, B_n}$ be an initial estimate of $P_0$ based on the training sample.

For any given $P \in \mathcal{M}$ and $\delta$, we consider the one-dimensional universal least favorable submodel $\{P_{\delta, \epsilon} : \epsilon \}$ through $P$. We define the model so that it passes through $P$ at the origin, i.e.$P_{0, \epsilon} = P$.

Consider the submodel $\{ \widehat{P}_{n,B_n,\delta, \epsilon} \}$ that passes through $\widehat{P}^0_{n, B_n}$ at the origin. Let $\epsilon_n$ be the MLE:
\begin{gather}
\epsilon_n \equiv \text{argmax} E_{B_n} P^1_{n, B_n} \log p_{n, B_n, \delta, \epsilon}.
\end{gather}

Let $\widehat{P}^*_{n, B_n, \delta} \equiv \widehat{P}_{n, B_n, \delta, \epsilon_n}$. We then have that $E_{B_n} P^1_{n, B_n} D_\delta^*(\widehat{P}^*_{n, B_n, \delta}) = 0$.

We define our new cross-validated TMLE estimator of $\Psi(P_0)$ as 
\begin{gather}
\widehat{\Psi}^{CV-TMLE}_n \equiv E_{B_n} \Psi_{\widehat{\delta}_n}(\widehat{P}^*_{n, B_n, \hat{\delta}_n}).
\end{gather}

We need to introduce additional assumptions needed for the analysis of this CV-TMLE estimator.

\medskip

\textbf{A9}. There exists $r_1 > 0$ such that
\begin{gather}
\left\lVert D^*_\delta(\widehat{P}^*_{n, B_n, \delta} ) - D^*_\delta(P_\infty) \right\rVert_{L_2(P_0)} = O_P (n^{-r_1} \delta^{-\gamma_{0, \infty}} ).
\end{gather}

\medskip

\textbf{A10}. $R_{\hat{\delta}_n}(\widehat{P}^*_{n, B_n, \delta}, P_0) = o_P(n^{-1/2}\sigma_\infty(\hat{\delta}_n)+ b_0 (\hat{\delta}_n))$.

\medskip

\textbf{A11}. There exists $L_\infty : \mathcal{O} \rightarrow \mathbb{R}$ such that
\begin{gather}
\left\lVert \delta^{\gamma_{0, \infty}} D^*_\delta(P_\infty) - L_\infty \right\rVert_{L_2(P_0)} = o_P(1).
\end{gather}

\medskip

\textbf{A12}. Consider the class of functions $\mathcal{F}_n \equiv \{ D^*_\delta(\widehat{P}^1_{n, B_n, \delta, \epsilon}) - D^*_\delta(P_\infty) : \delta, \epsilon \}$ and let $F_n$ be its envelope. Assume that 
\begin{gather}
\sup_{\Lambda} N(h|F_n|, \mathcal{F}_n, L_2(\Lambda)) = O(h^{-p}), \text{ for some integer } p > 0,
\end{gather}
where the sup is over all finitely discrete probability distributions.

\medskip

\paragraph{Discussion of the assumptions.} Given the one-dimensional nature of the family $\mathcal{F}_n$ in assumption \textbf{A11}, the covering number requirement should be very mild. Besides, we conjecture that assumptions \textbf{A9} and \textbf{A12} hold in the case where the approximating family is obtained by kernel smoothing. 

\medskip

We now state an asymptotic normality result for our CV-TMLE estimator.

\begin{theorem}\label{asymptotic_normality-CV-TMLE}
Assume that the conditions for lemma \ref{lemma-rate-estimators}, i.e. \textbf{A1} through \textbf{A4} hold. Assume \textbf{A9} through \textbf{A12}.
Then
\begin{gather}
\widehat{C}_{\sigma, n}^{-1} \sqrt{n}\hat{\delta}_n^{\hat{\gamma}_n} \left( \widehat{\Psi}_n^{CV-TMLE} - \Psi(P_0) \right) \xrightarrow{d} \mathcal{N}(0,1).
\end{gather}
\end{theorem}

\section{Examples}

\subsection{Estimation of a probability density function at a point}

We provide here a direct application of our estimators in the context of example 1, namely estimation of a p.d.f. at point.

\medskip

We remind the reader of the notations. We denote $P$ a probabibility distribution of the random variable $O$, which takes values in $\mathcal{O} \subset \mathbb{R}$. We denote $p$ the density of $P$ with respect to the Lebesgue measure. We denote $P_0$ the data-generating distribution of $O_1,...,O_n$. Our target parameter here is $\Psi(P_0) \equiv p_0(x)$, for some $x \in \mathcal{O}$. We consider smoothed parameters of the form $\Psi_{\delta}(P) = E_P\{ \delta^{-1}K((\cdot - x)/\delta)\}$, with $K$ a kernel as described above. As $\Psi_{\delta}$ is linear, the remainder $R_{\delta}$ in the first order expansion \eqref{first_order_expansion-smoothed} is zero. Recall that the canonical gradient of $\Psi_{\delta}$ at $P$ is given by $D^*_\delta(P) = \delta^{-1} K((\cdot - x)/\delta) - \delta^{-1} P K((\cdot - x)/\delta)$.

We will use the empirical probability distributions $P_{i,n}$, $i=1,2$, as initial estimators $\widehat{P}_{i,n}$, $i=1,2$. An initial estimator of $\Psi_{\delta}(P_0)$ is then given by $\Psi_{\delta}(P_{i,n}) \equiv P_{i,n} \{\delta^{-1} K((\cdot - x)/\delta)\}$.

It can easily be observed that $D^*_{\delta}(P_n) \xrightarrow{L_2(P_0)} D^*_{\delta}(P_0)$. Thus $D^*_{\delta, \infty} = D^*_\delta(P_0)$.

\medskip

Let us now examine the assumptions in this context.

It is easy to check that, under very mild assumptions (e.g. continuity of $p_0$ and $K$), we have that $\sigma_\infty(\delta) \sim C_{\sigma, 0, \infty} \delta^{-1/2}$, $\sigma'_\infty(\delta) \sim C_{\sigma', 0, \infty} \delta^{-3/2}$, $\sigma_\infty''(\delta) = O(\delta^{-5/2})$  and $P_0 (H_{0,\delta, \infty} - P_0 H_{0,\delta, \infty} )^2 \sim C_H \delta^{-3}$, for some positive constants $C_{\sigma, 0, \infty}$, $C_{\sigma', 0, \infty}$, and $C_H$. Kernel density estimation literature (see e.g. \cite{stone1984}) shows that if $p_0$ is $J_0$ times continously differentiable at $x_0$ and $K$ is a $J_K$-order kernel, then $b_0(\delta) \sim C_{b,0} \delta^{\min(J_0, J_K - 1)}$, $b_0'(\delta) \sim C_{b',0} \delta^{\min(J_0, J_K - 1) - 1}$ and $b_0''(\delta) = O\left(\delta^{\min(J_0, J_K - 2) - 1}\right)$, for some positive constants $C_{b, 0}$ and $C_{b',0}$. Therefore, $p_0$ being $J_0 \geq 1$ times continuously differentiable ensures that \textbf{A1} is satisfied. Note that this is just a sufficient condition.

As $R_{\delta} = 0$, \textbf{A3}, \textbf{A5}, \textbf{A7}, and \textbf{A8} are trivially verified. One also readily shows that $\sigma_n(\delta) = \sigma_\infty(\delta)$, for all $\delta$. Therefore \textbf{A2}, \textbf{A6} also hold. 

This proves the following corrolary of theorem \ref{asymptotic_normality_thm} and theorem \ref{MSE_optimality_theorem}.

\begin{corrolary}
Consider the setting and notations of example 1, recalled above.

Then, assumptions \textbf{A2} through \textbf{A8} are verified.

If one further assumes \textbf{A1}, then we have that
\begin{gather}
\widehat{C}_{\sigma, n}^{-1} \hat{\delta}_{\epsilon, n}^{\widehat{\gamma}_n} \left(l_{3,n} - l_{2,n}\right)^\frac{1}{2} \left(\widehat{\Psi}_n(\hat{\delta}_{\epsilon, n}) - \Psi(P_0) \right) \xrightarrow{d} \mathcal{N}(0,1).
\end{gather}

Also, the smoothing level selector $\hat{\delta}_{0,n}$ is asymptotically rate-optimal in the sense that
\begin{gather}
\frac{MSE_n(\widehat{\delta}_{0, n})}{MSE_n(\delta^*_{0,n})} \xrightarrow{P} K(p_2),
\end{gather}
where $K(p_2)$ is a constant that is a decreasing function of $p_2$.

Finally, the probability that the target parameter $\Psi(P_0) = p_0(x_0)$ belongs to the confidence interval
\begin{gather}
CI'_{\alpha, \epsilon, n} \equiv \left[\widehat{\Psi}_n\left(\hat{\delta}_{0,n}\right) \mp q_{1-\alpha/2} \frac{\widehat{C}_{\sigma, n}}{\left(l_{3,n} - l_{2,n} \right)^{\frac{1}{2} - \left(\hat{r}_n + \epsilon\right) \hat{\gamma}_n}} \right]
\end{gather}
converges to one as the sample size $n$ tends to infinity.

\end{corrolary}

\subsection{Estimation of a mean counterfactual outcome}

We illustrate here our method in the case of the estimation of a counterfactual mean outcome, under known treatment mechanism, as presented in example 2.

\medskip

The target parameter mapping here is defined, for all $P \in \mathcal{M}$ as $\Psi(P) \equiv E_P\left[E_P[Y | A = 1, W] \right]$. We consider smoothed parameters of the form $\Psi_{\delta}(P) \equiv E_P \left[\frac{g_0(1|W)}{g_{0, \delta}(1|W)} E_P\left[Y | A =1, W\right] \right]$. Note that $\Psi_{\delta}$ is linear in $P$ and thus the remainder term $\Psi_{\delta}$ in the first order expansion \eqref{first_order_expansion-smoothed} is zero.

Recall that for $P \in \mathcal{M}$ the likelihood $p \equiv \frac{dP}{d\mu}$ factors as $p = q_Y q_W g_0$, where $q_Y$ is the conditional likelihood of the outcome given the treatment value and baseline covariate, $q_W$ is the likelihood of the baseline covariates and $g_0$ is the previously introduced conditional likelihood of treatment given the baseline covariates. 

Observe that $\Psi_{\delta}(P)$ only depends on $P$ through $Q_W$ and $\bar{Q}(1,W) \equiv E_{Q_Y}[Y | A = 1, W]$, and that $D^*_\delta(P)$ only depends on $P$ through $Q_W$, $g_0$ and $\bar{Q}(1,W)$. Therefore, in the definition of our initial estimator $\widehat{P}$, we need only specify estimators of $Q_W$ and of $\bar{Q}$. We will use the empirical distribution $Q_{W,n}$ of $W_1,...,W_n$ as initial estimate of $Q_W$. We will estimate the regression function $\bar{Q}$ with a kernel regression estimate $\widehat{\bar{Q}}_n$. 
We will use a kernel regression estimator that is uniformly consistent with respect to its limit.

\medskip

Let us now examine the assumptions of our method in this context. As $R_\delta = 0$, \textbf{A3} and \textbf{A5} are trivially verified.

Let us turn to the assumptions \textbf{A2} and \textbf{A6}. The following results prove useful.

\begin{lemma}\label{lemma-preA2-example2}
Consider the setting of example 2, recalled above.
We have that
\begin{gather}
|\sigma_{1,n}(\delta) - \sigma_\infty(\delta)| \leq \delta^{-\frac{1}{2}} \|\hat{\bar{Q}}_n - \bar{Q}_\infty \|_{L_2(P_0)}.
\end{gather}
\end{lemma}

Therefore, as long as
$\|\hat{\bar{Q}}_n - \bar{Q}_\infty \|_{L_2(P_0)} = O_P \left(n^{-r_Q}\right)$ for some $r_Q > 0$, lemma \ref{lemma-preA2-example2} ensures that assumption \textbf{A2} holds. This is a very mild condition.

\begin{lemma}\label{lemma-sigma_n_sigma_inf-EY1-example}
Assume 
\begin{equation}\label{uniform_consistency_of_ratio}
\left\lVert \frac{\left(\hat{\bar{Q}}_n - \bar{Q}_\infty\right)^2}{\bar{Q}_0 (1 - \bar{Q}_0 ) + \left(\bar{Q}_\infty - \bar{Q}_0 \right)^2} \right\rVert_{L_\infty(P_0)} \xrightarrow{P} 0.
\end{equation}

Then for any non-negative sequence $\delta_n$ that converges to zero, we have that 
$\sigma_n(\delta_n) \sim_P \sigma_\infty(\delta_n)$, i.e. assumption \textbf{A6} is verified.
\end{lemma}

Under certain conditions, notably on the rate of the bandwidth \textit{(add citation!!!!! + refer to new tech report about uniform consistency of HAL)}, kernel regression estimates are uniformly consistent, i.e. $\| \hat{\bar{Q}}_n - \bar{Q}_0 \|_{L_\infty(P_0)}$ converges to zero in probability. If one further assumes that $\bar{Q}_0$ is bounded away from $0$ and $1$, this ensures that assumption $\eqref{uniform_consistency_of_ratio}$ is satisfied.

This discussion thus proves the following corrolary of theorem \ref{asymptotic_normality_thm} and theorem \ref{MSE_optimality_theorem}.

\begin{corrolary}
Consider the setting and notations of example 2, recalled above.

Then, assumptions \textbf{A3}, \textbf{A5}, \textbf{A7} and \textbf{A8} are verified.

Assume 
\begin{gather}
\left\lVert\hat{\bar{Q}}_n - \bar{Q}_\infty\right\rVert_{L_2(P_0)} = O_P \left(n^{-r_Q}\right),
\end{gather}
for some $r_Q > 0$, and that
\begin{gather}
\left\lVert \frac{\left(\hat{\bar{Q}}_n - \bar{Q}_\infty\right)^2}{\bar{Q}_0 (1 - \bar{Q}_0 ) + \left(\bar{Q}_\infty - \bar{Q}_0 \right)^2} \right\rVert_{L_\infty(P_0)} = o_P(1).
\end{gather}
Then assumptions \textbf{A2} and \textbf{A6} are verified. 

If one further assumes \textbf{A1}, we then have that
\begin{gather}
\widehat{C}_{\sigma, n}^{-1} \hat{\delta}_{\epsilon, n}^{\widehat{\gamma}_n} \left(l_{3,n} - l_{2,n}\right)^\frac{1}{2} \left(\widehat{\Psi}_n(\hat{\delta}_{\epsilon, n}) - \Psi(P_0) \right) \xrightarrow{d} \mathcal{N}(0,1).
\end{gather}

Also, the smoothing level selector $\hat{\delta}_{0,n}$ is asymptotically rate-optimal in the sense that
\begin{gather}
\frac{MSE_n(\widehat{\delta}_{0, n})}{MSE_n(\delta^*_{0,n})} \xrightarrow{P} K(p_2),
\end{gather}
where $K(p_2)$ is a constant that is a decreasing function of $p_2$.

Finally, the probability that the target parameter $\Psi(P_0) = E_{P_0}E_{P_0}[Y | A = 1, W]$ belongs to the confidence interval
\begin{gather}
CI'_{\alpha, \epsilon, n} \equiv \left[\widehat{\Psi}_n\left(\hat{\delta}_{0,n}\right) \mp q_{1-\alpha/2} \frac{\widehat{C}_{\sigma, n}}{\left(l_{3,n} - l_{2,n} \right)^{\frac{1}{2} - \left(\hat{r}_n + \epsilon\right) \hat{\gamma}_n}} \right]
\end{gather}
converges to one as the sample size $n$ tends to infinity.

\end{corrolary}

\subsection{Estimation of dose-response curve at a fixed dose value}

We demonstrate here our method in the case of a the estimation of the dose-response curve at a fixed dose value $a_0 \in [0, 1]$.

Recall that our target parameter is defined for all $P \in \mathcal{M}$ by $\Psi_{a_0}(P) \equiv E_P E_P \left[Y | A = a_0, W\right]$. Our approximating family is defined by the kernel smoothed parameters $\Psi_{a_0}(P) \equiv \int_a K_{\delta, a_0}(a) \Psi_{a}(P) da$, where $K_{\delta, a_0}(a) = \delta^{-1} K((a - a_0) / \delta )$. Note that $\Psi_{a_0, \delta}$ is linear in $P$ and thus the remainder term $R_\delta$ in the first order expansion \eqref{first_order_expansion-smoothed} is zero.

Recall that for $P \in \mathcal{M}$ the likelihood $p \equiv \frac{dP}{d\mu}$ factors as $p = q_Y q_W g_0$, where $q_Y$ is the conditional likelihood of the outcome given the treatment value and baseline covariate, $q_W$ is the likelihood of the baseline covariates and $g_0$ is the previously introduced conditional likelihood of treatment given the baseline covariates.

Observe that $\Psi_{a_0, \delta}(P)$ and $D^*_{a_0, \delta}(P)$ only depends on $P$ through $Q_W$ and $\bar{Q}(a, W) = E_{Q_Y} \left[Y | A = a, W \right]$. Therefore, in the definition of our initial estimator $\hat{P}$, we only need to specify estimators of $Q_W$ and $\bar{Q}$. We will use the empirical distribution $Q_{W, n}$ of $W_1,...,W_n$ as initial estimator of $Q_W$. We will use a nonparametric estimator of $\bar{Q}$ whose required properties will be made clear below.

\medskip

Let us now examine the assumptions of our method in this context. As $R_\delta = 0$, \textbf{A3} and \textbf{A4} are trivially verified.

Let us turn to the assumptions \textbf{A2} and \textbf{A6}. The following lemma proves useful.

\begin{lemma}\label{lemma-preA2-example3}
Consider the dose response curve of example 3, recalled in this section.

Assume that $\|g_0^{-1}\|_{L_\infty(P_0)} < \infty$.

Then, we have that
\begin{gather}
|\sigma_{1, n}(\delta) - \sigma_\infty(\delta)| \leq O_P\left(\delta^{-\frac{1}{2}} \|\hat{\bar{Q}}_n - \bar{Q}_\infty \|_{L_\infty(P_0)} \right).
\end{gather}
\end{lemma}

Therefore, if $\hat{\bar{Q}}_n$ converges uniformly with a polynomial rate, i.e. if $\|\hat{\bar{Q}}_n - \bar{Q}_\infty\|_{L_\infty(P_0)} \leq O_P\left(n^{-r_Q}\right)$, for some $r_Q > 0$, lemma \ref{lemma-preA2-example3} guarantees that assumption \textbf{A2} is satisfied. Note that the assumption that $\|g_0^{-1}\|_{L_\infty(P_0)} < \infty$ is the so-called \textit{positivity} assumption from causal inference (see e.g. \cite{petersen2011}).

Lemma \ref{lemma-sigma_n_sigma_inf-EY1-example}
above also holds in the context of this section. (We provide a separate proof for each of these two examples in the appendix). Therefore, if one assumes that $\bar{Q}_0$ is bounded away from 0 and 1, using a uniformly consistent estimator of $\bar{Q}_0$ is enough to ensure \textbf{A6} is verified.

This discussion proves the following corrolary of theorem \ref{asymptotic_normality_thm} and \ref{MSE_optimality_theorem}.

\begin{corrolary}
Consider the setting and notations of example 2, recalled above.

Then, assumptions \textbf{A3}, \textbf{A5}, \textbf{A7} and \textbf{A8} are verified.

Assume $\| g_0^{-1} \|_{L_\infty(P_0)} < \infty$, $
\|\hat{\bar{Q}}_n - \bar{Q}_\infty\|_{L_\infty(P_0)} = O_P(n^{-r_Q})$, for some $r_Q > 0$, and that
\begin{gather}
\left\lVert \frac{\left(\hat{\bar{Q}}_n - \bar{Q}_\infty\right)^2}{\bar{Q}_0 (1 - \bar{Q}_0 ) + \left(\bar{Q}_\infty - \bar{Q}_0 \right)^2} \right\rVert_{L_\infty(P_0)} = o_P(1).
\end{gather}
Then assumptions \textbf{A2} and \textbf{A6} are verified. 

If one further assumes \textbf{A1}, we then have that
\begin{gather}
\widehat{C}_{\sigma, n}^{-1} \hat{\delta}_{\epsilon, n}^{\widehat{\gamma}_n} \left(l_{3,n} - l_{2,n}\right)^\frac{1}{2} \left(\widehat{\Psi}_n(\hat{\delta}_{\epsilon, n}) - \Psi(P_0) \right) \xrightarrow{d} \mathcal{N}(0,1).
\end{gather}

Also, the smoothing level selector $\hat{\delta}_{0,n}$ is asymptotically rate-optimal in the sense that
\begin{gather}
\frac{MSE_n(\widehat{\delta}_{0, n})}{MSE_n(\delta^*_{0,n})} \xrightarrow{P} K(p_2),
\end{gather}
where $K(p_2)$ is a constant that is a decreasing function of $p_2$.

Finally, the probability that the target parameter $\Psi(P_0) = E_{P_0}E_{P_0}[Y | A = 1, W]$ belongs to the confidence interval
\begin{gather}
CI'_{\alpha, \epsilon, n} \equiv \left[\widehat{\Psi}_n\left(\hat{\delta}_{0,n}\right) \mp q_{1-\alpha/2} \frac{\widehat{C}_{\sigma, n}}{\left(l_{3,n} - l_{2,n} \right)^{\frac{1}{2} - \left(\hat{r}_n + \epsilon\right) \hat{\gamma}_n}} \right]
\end{gather}
converges to one as the sample size $n$ tends to infinity.

\end{corrolary}

\section{Simulation results for the dose-response curve example}

We consider the following example of data-generating distribution, which we took from \citep{edwardKennedy2016}.

\begin{gather}
L \equiv (L_1, L_2, L_3, L_4) \sim \mathcal{N}(0, I_4)\\
\lambda(L) = expit(-0.8 + 0.1 L_1 + 0.1 L_2 - 0.1 L_3 + 0.2 L_4)\\ \\
A = Beta(\lambda(L), 1 - \lambda(L))\\ \\
\mu(L, A) \sim expit(1 + 0.2 L_1 + 0.2 L_2 + 0.3 L_3 - 0.1 L_4 +\\
 20 A (0.1 - 0.1 L_1 + 0.1 L_3 -0.13^2 (20 A)^2)\\ \\
Y \sim Bernouilli(\mu(L, A))
\end{gather}

We target the causal dose response curve at $a_0 = 0.15$, i.e. we want to infer $\Psi_{0.15}(P_0)$.

We compare our smoothing level selector to alternative deterministic smoothing rates. Let us first expose the rationale behind our choice of competing smoothing rates.

One can readily prove that the optimal smoothing rate depends on the smoothness of $a \mapsto \Psi_a(P_0)$ at $a_0$ and on properties of the kernel $K$. Concretely, one can show that if $a \mapsto \Psi_a(P_0)$ is $J$ times differentiable at $a_0$, and $K$ is orthogonal to polynomials of degree smaller than or equal to $J-1$, the optimal smoothing level $h^*_n \sim C^* n^{-1/(2J + 1)}$. (For instance, if $\Psi_a(P_0)$ is twice differentiable at $a_0$ and $K$ is orthogonal to all polynomials of degree at most $1$, the optimal smoothing rate is $n^{-1/5}$.)

These considerations motivate us to consider competing deterministic smoothing rates of the form $C n^{-1/5}$, $C n^{-1/7}$, $C n^{-1/9}$, with $C$ a positive constant. We then use the same type of single-fold, three-splits cross-validated estimatars of the smoothed parameters, as defined above, with these competing smoothing rates.

In addition to our single-split cross-validated one-step estimator we also used in this simulation a $V$-fold cross-validated one-step. We also computed the Targeted Maximum Likelihood Estimates presented in section \ref{CV-TMLE-section}.

We report below plots of the mean squared error (with respect to $\Psi(P_0)$) against sample size, for all of these estimators. We also present estimates of the optimal smoothing rate and coverage rates of the ensuing confidence intervals.

\begin{figure}[!htb]
\centering
\includegraphics[width = 14cm, height=9.33cm]{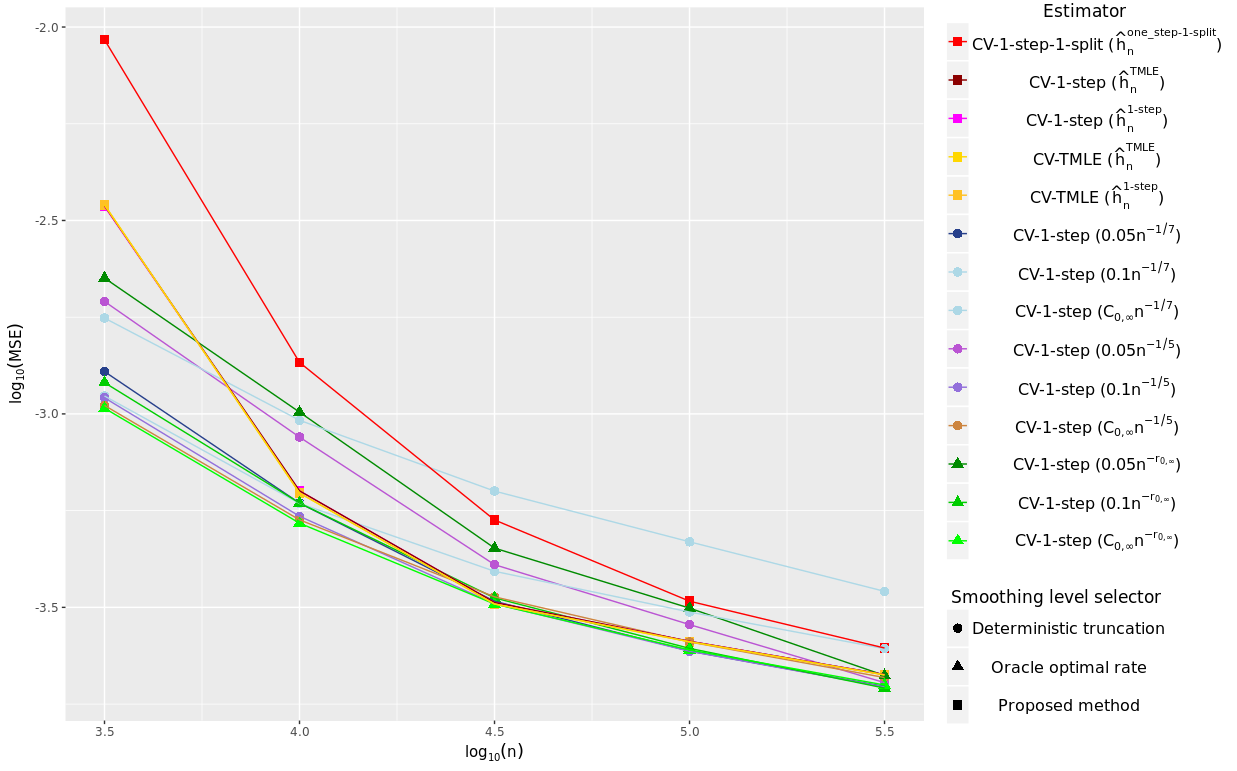}
\caption{An illustration of the performance in mean squared error (w.r.t. $\Psi(P_0)$) of our method compared to 5-fold cross-validated one-step estimators with deterministic smoothing level. The competing deterministic smoothing levels are of the form $C n^{-r}$, with $C \in \{0.05, 0.1, C_{0, \infty} \}$ and $r \in \{1/5, 1/7, r_{0, \infty} \}$.
Each point in the plot is obtained by averaging the squared error w.r.t. $\Psi(P_0)$ over 315 i.i.d. datasets sampled from the data-generating distribution described above. 
Analytic derivation show that the optimal smoothing rate is $n^{-\frac{1}{5}}$. However, Monte-Carlo simulations show that for the sample size range considered (i.e. from $10^{3.5}$ to $10^{5.5}$), the optimal smoothing level is $\approx 0.132 n^{-0.183}$.) The above plot shows that the choice of smoothing rate $n^{-1/7}$ can prove much less efficient than the oracle. Our method seems to asymptotically perform on par with the oracle optimal smoothing level.}
\end{figure}

\begin{figure}[!htb]
\centering
\includegraphics[width = 14cm, height=9.33cm]{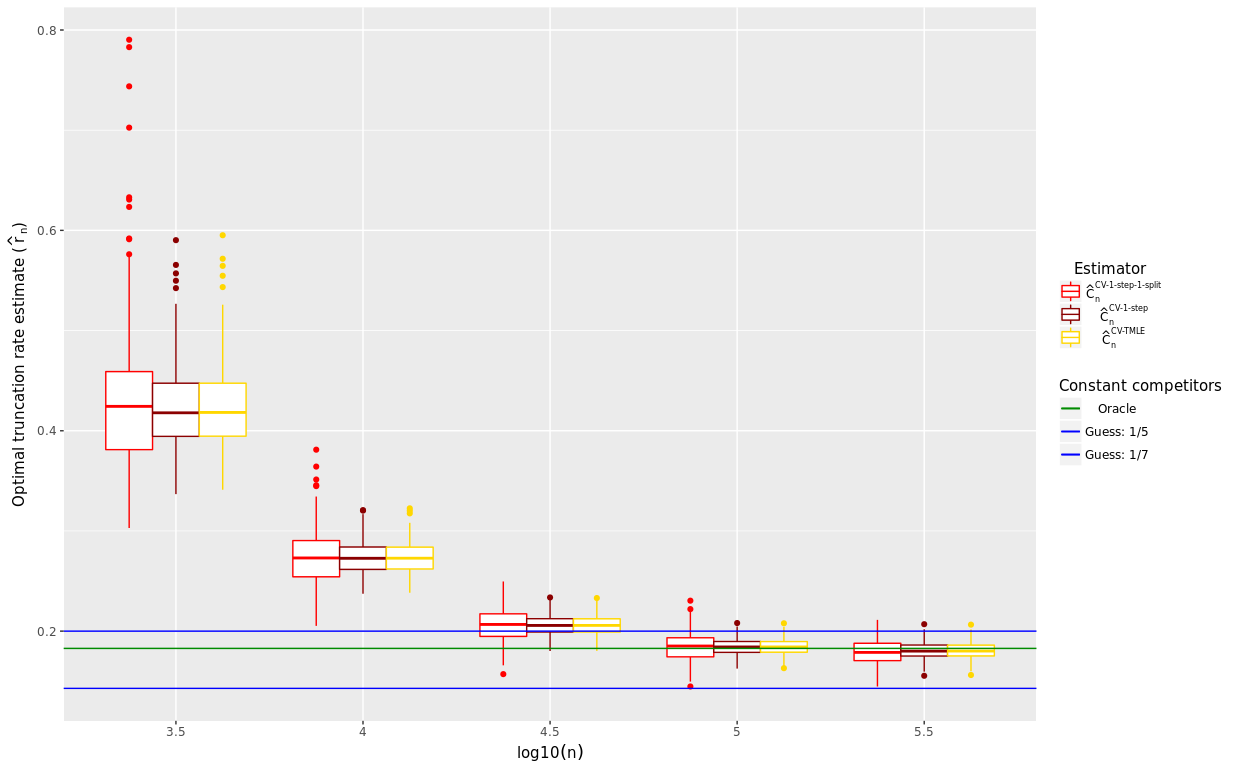}
\end{figure}

\begin{figure}[!htb]
\centering
\includegraphics[width = 14cm, height=9.33cm]{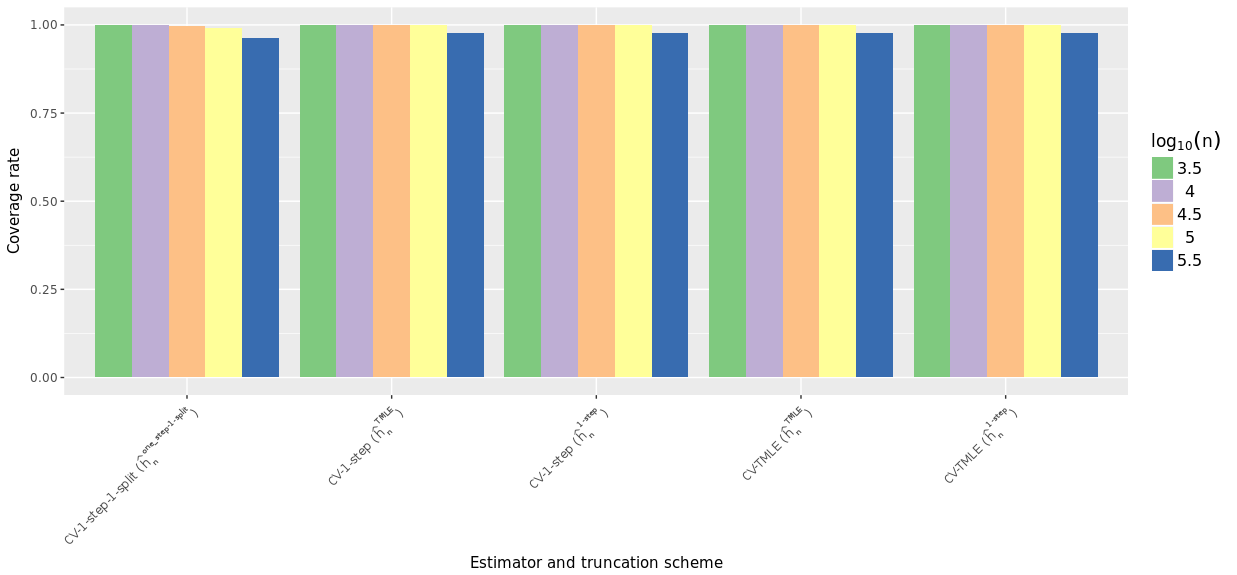}
\end{figure}

We expect that the procedure presented in \cite{edwardKennedy2016} would have performed on par with ours on this specific example. Indeed, their procedure finds the smoothing rate that is optimal in terms in mean integrated squared error with respect to the dose response curve (i.e. the integral w.r.t. $a$ of the squared difference between the estimated  curve and the true curve). Since in this example, the curve is at least twice differentiable everywhere, their work shows that their estimated smoothing rate when using a Gaussian kernel is asymptotically $n^{-1/5}$. However their results do no guarantee their procedure is optimal if the smoothness of the curve varies with $a$. In the case that it is not differentiable only at $a_0$, we expect that their procedure would have used a smoothing rate close to $n^{-1/5}$ as it would be dictated mostly by the smoothness rest of the curve.

This motivate us to perform simulations in a case where $a \mapsto \Psi_a(P_0)$ is not differentiable at $a_0$. We consider a data-generating distribution that implies a cusp in the curve at $a_0$. We obtain this distribution from the one specified in the previous example, by replacing $\mu(A, L)$ by

\begin{gather}
\mu(L, A) \sim expit(1 + 0.2 L_1 + 0.2 L_2 + 0.3 L_3 - 0.1 L_4 +\\
 20 A (- 0.1 L_1 + 0.1 L_3 -0.13^2 (20 A)^2) + 5 \times \text{cusp}(A)),
\end{gather}
where $\text{cusp}(a) = I(a \leq 0.15) a + I(a > 0.15) (0.15 - 2 (a - 0.15))$.

We present plots of the mean squared error (w.r.t. $\Psi(P_0)$) against sample size, and of coverage rates.

\begin{figure}[!htb]
\centering
\includegraphics[width = 14cm, height=9.33cm]{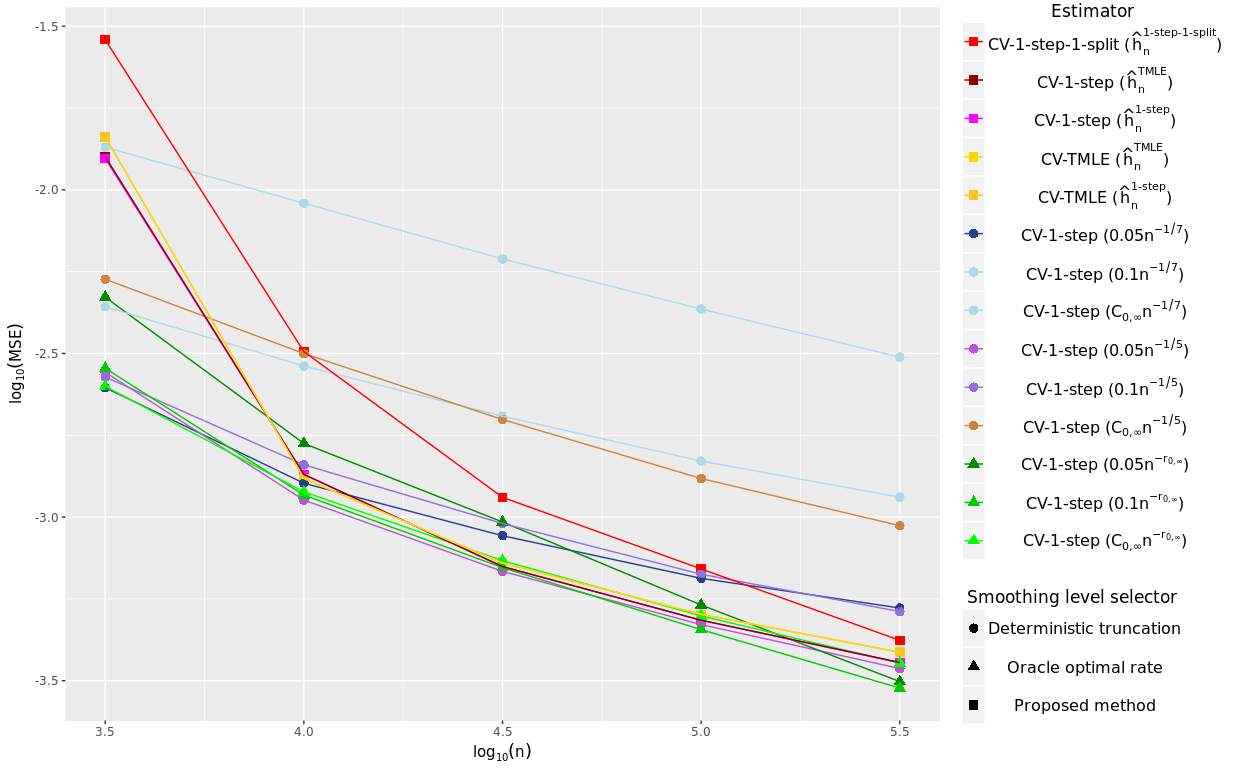}
\end{figure}

\begin{figure}[!htb]
\centering
\includegraphics[width = 14cm, height=9.33cm]{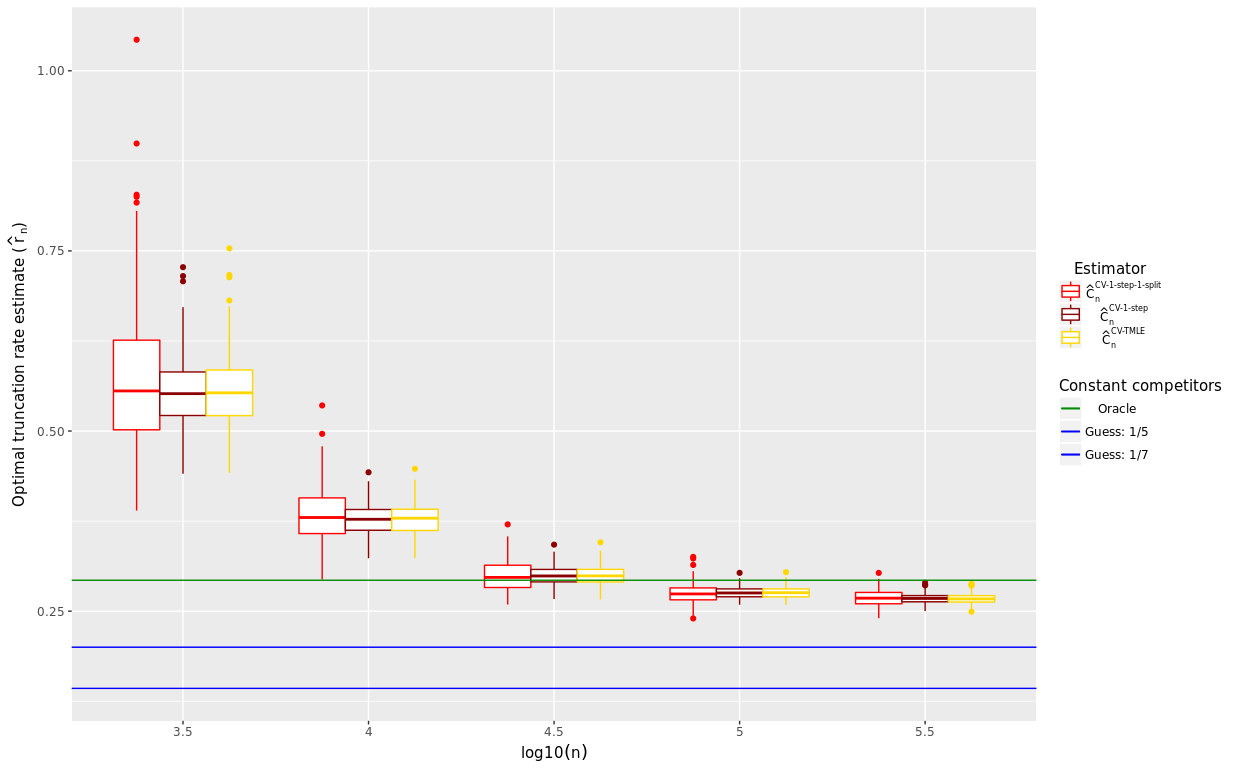}
\includegraphics[width = 14cm, height=9.33cm]{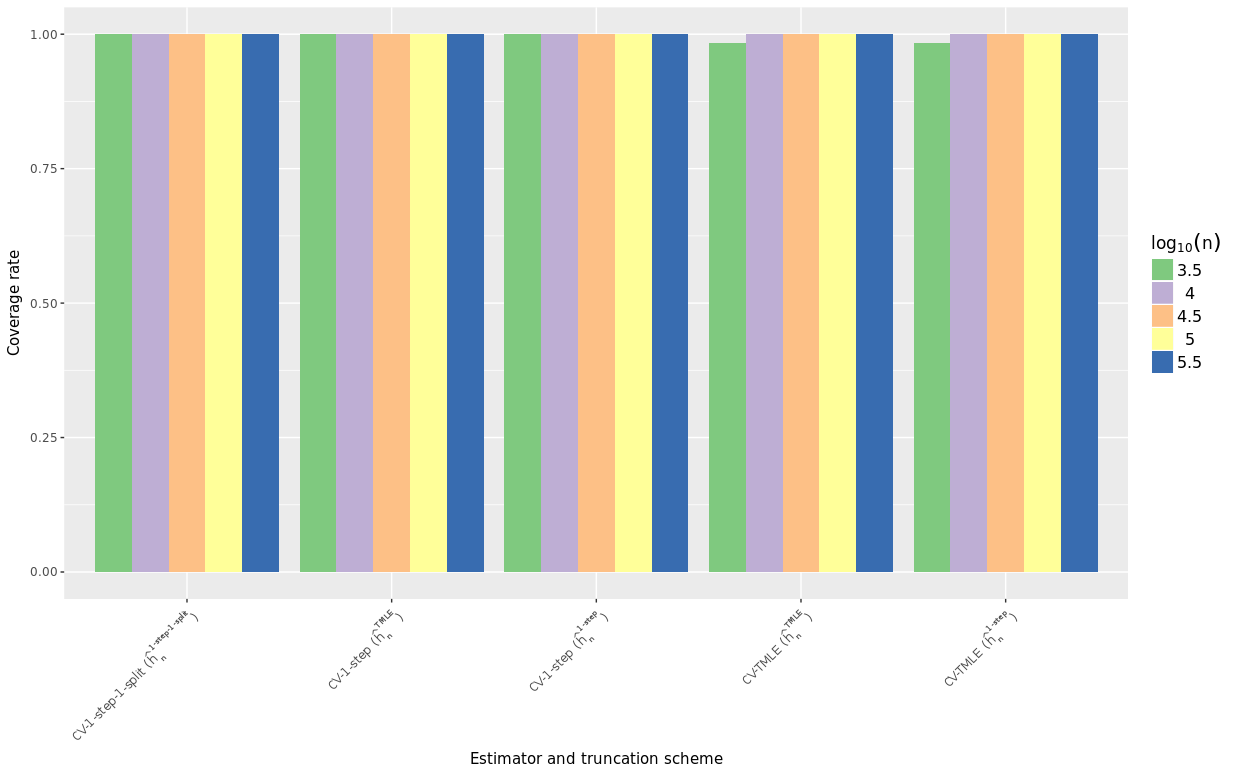}
\end{figure}

As can be observed in the above plots, in both examples, our procedure outperforms deterministic (but informed) choices of the smoothing level.

\FloatBarrier

\section{Discussion}

We have presented a general method to select the optimal smoothing level $\hat{\delta}_n$ in a variety of non-parametric inference problems. Under some assumptions, our one-step estimator at $\hat{\delta}_n$ is asymptotically normally distributed and we show how to construct confidence intervals. Under some additional assumptions, it is rate-optimal among the class of estimators of the form $\widehat{\Psi}_n(\delta_n)$, where $\widehat{\Psi}_n(\delta)$ is an asymptotically efficient, double robust estimator of $\Psi_\delta(P_0)$.

We have shown that using widely available, off-the-shelf, initial estimators of $P_0$ make our assumptions hold in the three concrete examples we 
considered. Simulations demonstrated the practical performance of our method in the dose response curve example.

We concede, however, that example 2 (estimation of $EY_1$ knowing the treatment mechanism) is likely of little utility in practice. Indeed, the only practical situation we have in mind where the treatment meachanism is known is when it was set by the researcher in advance. There should be no point in setting so small that we would have positivity issues. The standard practice is on the contrary to balance treatment and control groups, i.e. to perform a randomized controlled trial. This is why we dedicate a forthcoming article to the situation where the treatment mechanism is unknown and thus estimate.

Besides, we have not provided much guidance in how to set the slow sequences $\tilde{\delta}_{1,n}$ and $\tilde{\delta}_{2,n}$. Taking them too slow makes the assumptions very likely to hold but impairs finite sample performance. We found that plotting $\log \widehat{b'}_{2,n}(\delta)$ against $\log \delta$ usually reveals a range of values of $\delta$ where the plot is linear. Taking $\tilde{\delta}_{i,n}$, $i = 1,2$ in this range yields very good practical performance. This is actually how we chose $\tilde{\delta}_{i,n}$, $i = 1,2$ in the simulation presented above.

\bibliographystyle{imsart-nameyear}
\bibliography{biblio}

\section{Appendix}

In this appendix we provide proofs of the asymptotic properties of our estimators (cross-validated single-split one-step estimator and CV-TMLE), and proofs that the assumptions of our general theorems are satisfied in our three examples (p.d.f. at a point, mean counterfactual outcome, and causal dose-response curve).

\subsection{Asymptotic analysis of the rate estimators}

\begin{lemma}\label{lemma-diff_sigma_hat-sigma_inf}
Assume \textbf{A1}, \textbf{A3} and \textbf{A4}. Then there exist $k_4 > 0$, $\kappa_4 > 0$ and $r^+ > 0$ such that for any sequence $\tilde{\delta}_n$ that converges to zero slower than $n^{-r^+}$,
\begin{equation}
\widehat{\sigma}_{2,n}^2(\tilde{\delta}_n) - \sigma^2_\infty(\tilde{\delta}_n) = O_P (\tilde{\delta}_n^{-k_4} \left(l_{2,n} - l_{1,n} \right)^{-\kappa_4}).
\end{equation}
\end{lemma}

\begin{proof}
Applying Pythagoras yields
\begin{gather}
P_{2,n} \left(D^*_\delta(\widehat{P}_{1,n}) - P_{2,n} D^*_\delta(\widehat{P}_{1,n})\right)^2 = P_{2,n} \left(D^*_\delta(\widehat{P}_{1,n}) - P_0 D^*_\delta(\widehat{P}_{1,n})\right)^2\\
- \left(P_{2,n} D^*_\delta(\widehat{P}_{1,n}) - P_0 D^*_\delta(\widehat{P}_{1,n}) \right)^2.
\end{gather}

Therefore, recalling the definitions of $H_{0, \tilde{\delta}_n}(\widehat{P}_{1,n})$, $H_{0, \infty, \tilde{\delta}_n}$, $\sigma_n(\tilde{\delta}_n)$, and $\sigma_\infty(\tilde{\delta}_n)$,
\begin{gather}
\widehat{\sigma}_{2,n}^2(\delta) - \sigma^2_\infty(\delta) = \left(P_{2,n} - P_0\right) H_{0, \delta}(\widehat{P}_{1,n}) + \sigma^2_{1,n}(\tilde{\delta}_n) - \sigma^2_\infty(\tilde{\delta}_n) \\
- \left(\left(P_{2,n} - P_0\right) D^*_\delta(\widehat{P}_{1,n})\right)^2.\label{diff_sigma_hat_sigma_inf}
\end{gather}

From assumption \textbf{A2}, for $\tilde{\delta}_n$ slow enough, $\sigma_n(\tilde{\delta}_n) \sim_P \sigma_\infty(\tilde{\delta}_n)$. Therefore, applying the central limit theorem for triangular arrays yields that
\begin{gather}
\left(P_{2,n} - P_0\right) D^*_\delta(\widehat{P}_{1,n}) = O_P (\sigma_\infty(\tilde{\delta}_n) n^{-1/2}) \\
= O_P (\tilde{\delta}_n^{-k_1} \left(l_{2,n} - l_{1,n} \right)^{-1/2} ).\label{emp_process_in_D-sigma_hat-proof}
\end{gather}

From assumption \textbf{A4}, for $\tilde{\delta}_n$ slow enough, $P_0 (H_{0, \tilde{\delta}_n}(\widehat{P}_{1,n}) - P_0H_{0, \tilde{\delta}_n}(\widehat{P}_{1,n}))^2 \allowbreak  \sim_P P_0 (H_{0, \tilde{\delta}_n, \infty} - P_0 H_{0, \tilde{\delta}_n, \infty})^2$. Therefore, normalizing the first empirical process term in \eqref{diff_sigma_hat_sigma_inf} by $(P_0 (H_{0, \tilde{\delta}_n}(\widehat{P}_{1,n}) - P_0 H_{0, \tilde{\delta}_n}(\widehat{P}_{1,n}))^2)^{1/2}$, applying the central limit theorem for triangular arrays, and then using assumption \textbf{A1} yields
\begin{gather}
\left(P_{2,n} - P_0\right) H_{0, \tilde{\delta}_n} (\widehat{P}_{1,n}) = O_P(\tilde{\delta}_n^{-\eta_{0, \infty}} (l_{2,n} - l_{1,n})^{-1/2}). \label{emp_process_in_H-sigma_hat-proof}
\end{gather}

Finally, recall that assumption \textbf{A2} states that, for $\tilde{\delta}_n$ slow enough,
\begin{gather}
\sigma_\infty^2(\tilde{\delta}_n) - \sigma_{1,n}^2(\tilde{\delta}_n) = O_P(\tilde{\delta}^{-k_2} \left(l_{2,n} - l_{1,n} \right)^{-\kappa_2} ). \label{A2-recalled}
\end{gather}

Therefore, injecting \eqref{emp_process_in_D-sigma_hat-proof}, \eqref{emp_process_in_H-sigma_hat-proof}, and \eqref{A2-recalled} in \eqref{diff_sigma_hat_sigma_inf} yields
\begin{equation}
\hat{\sigma}_{2,n}^2(\tilde{\delta}_n) - \sigma_\infty^2(\tilde{\delta}_n) = O_P (\tilde{\delta}_n^{-k_4} \left(l_{2,n} - l_{1,n} \right)^{-\kappa_4}),
\end{equation}

with $k_4 \equiv \min(k_1, k_2, \eta_{0, \infty})$ and $\kappa_4 \equiv \min(1/2, \kappa_2)$.
\end{proof}

\medskip

\begin{lemma}\label{lemma-diff_b_prime_hat-b_prime} Assume \textbf{A1} through \textbf{A5}. Then there exists $k_5 > 0$, $\kappa_5 > 0$, and $r^+ > 0$ such that, for any sequence $\tilde{\delta}_n$ that converges to zero slower than $n^{-r^+}$,
\begin{equation}
\widehat{b'}_{2,n}(\tilde{\delta}_n) - b'_0(\tilde{\delta}_n) = O_P (\tilde{\delta}_n^{-k_5} \left(l_{2,n} - l_{1,n} \right)^{-\kappa_5} ).
\end{equation}

\end{lemma}

\begin{proof}
Observe that
\begin{gather}
\widehat{b'}_n(\tilde{\delta}_n) - b'_0(\tilde{\delta}_n) = \Delta_n^{-1} \big\{ \Psi_{\tilde{\delta}_n + \Delta_n}(P_0) - \Psi_{\tilde{\delta}_n}(P_0) \\
+ \left(P_{2,n} - P_0 \right) \left(D^*_{\tilde{\delta}_n + \Delta_n} (\widehat{P}_{1,n}) - D^*_{\tilde{\delta}_n}(\widehat{P}_{1,n})\right)\\
+ R_{\tilde{\delta}_n + \Delta_n}(\widehat{P}_{1,n}, P_0) -  R_{\tilde{\delta}_n}(\widehat{P}_{1,n}, P_0) \big\}\\
= b'_0(\bar{\delta}_{1,n}) - b'_0(\tilde{\delta}_n) + \left(P_{2,n} - P_0\right) \frac{D^*_{\tilde{\delta}_n + \Delta_n} (\widehat{P}_{1,n}) - D^*_{\tilde{\delta}_n}(\widehat{P}_{1,n})}{\Delta_n} \\
+ \frac{\partial R_\delta(\widehat{P}_{1,n}, P_0)}{\partial \delta}\bigg|_{\delta = \bar{\delta}_{2,n}},\label{main_difference-lemma_b_hat}
\end{gather}

where $\bar{\delta}_{1,n}, \bar{\delta}_{2,n} \in \left[\tilde{\delta}_n, \tilde{\delta}_n + \Delta_n \right]$.

As seen in the proof of lemma \ref{lemma-diff_sigma_hat-sigma_inf}, $\left(P_{2,n} - P_0\right) D^*_{\tilde{\delta}_n}(\widehat{P}_{1,n}) = O_P(\tilde{\delta}_n^{-\gamma_{0, \infty}} (l_{2,n} - l_{1,n} )^{-1/2} )$. Similarly $\left(P_{2,n} - P_0\right) D^*_{\tilde{\delta}_n + \Delta_n}(\widehat{P}_{1,n}) = O_P(\tilde{\delta}_n^{-\gamma_{0, \infty}}(l_{2,n} - l_{1,n} )^{-1/2})$. Therefore
\begin{equation}\label{finite_diff_emp_process-proof_b'_hat}
\left(P_{2,n} - P_0\right) \frac{D^*_{\tilde{\delta}_n + \Delta_n} (\widehat{P}_{1,n}) - D^*_{\tilde{\delta}_n}(\widehat{P}_{1,n})}{\Delta_n} = O_P(\tilde{\delta}_n^{-\gamma_{0, \infty}} (l_{2,n} - l_{1,n})^{-\frac{1}{4}} ).
\end{equation}

Besides $|b'_0(\bar{\delta}_{1,n}) -b'_0(\tilde{\delta}_n)| \leq \Delta_n b''_0(\bar{\delta}_{3,n})$, for some $\bar{\delta}_{3,n} \in \left[\tilde{\delta}_n, \bar{\delta}_{1,n}\right]$. Thus, from assumption \textbf{A1},
\begin{equation}
b'_0(\bar{\delta}_{1,n}) -b'_0(\tilde{\delta}_n) = O_P((l_{2,n} - l_{1,n})^{-\frac{1}{4}} \tilde{\delta}_n^{-k_0}).\label{b'0_diff}
\end{equation}

Finally, from assumption \textbf{A5}, $(\partial R_{\delta}(\hat{P}_{1,n}, P_0 ) /\partial \delta ) |_{\delta = \bar{\delta}_{1,n}} = O_P(\tilde{\delta}_n^{-k_3}(l_{2,n} - l_{1,n} )^{-\kappa_3}).$ Injecting this latter equation, \eqref{b'0_diff} and \eqref{finite_diff_emp_process-proof_b'_hat} into \eqref{main_difference-lemma_b_hat}, we obtain
\begin{gather}
\widehat{b'}_n(\tilde{\delta}_n) - b'_0(\tilde{\delta}_n) = O_P ( \tilde{\delta}_n^{-k_5} (l_{2,n} - l_{1,n} )^{-\kappa_5}),
\end{gather}
where $k_5 \equiv \min(k_0, k_3, \gamma_{0, \infty})$ and $\kappa_5 \equiv \min(\kappa_3, 1/4)$.
\end{proof}

\medskip

\begin{lemma}\label{lemma-sigma_prime_hat-sigma_prime}
Assume \textbf{A1} through \textbf{A4}. Then there exist $k_6 > 0$, $\kappa_6$, $r^+ > 0$ such that, for any sequence $\tilde{\delta}_n$ that converges to zero slower than $n^{-r^+}$,
\begin{equation}
\widehat{\sigma'}_{2,n}(\tilde{\delta}_n) - \sigma'_\infty(\tilde{\delta}_n) = O_P (\tilde{\delta}_n^{-k_6}\left(l_{2,n} - l_{1,n}\right)^{-\kappa_6}).
\end{equation}
\end{lemma}

\begin{proof}
Observe that
\begin{gather}
\widehat{\sigma'}_{2,n}(\tilde{\delta}_n) - \sigma'_\infty(\tilde{\delta}_n) \\
=\Delta_n^{-1} \big\{ \widehat{\sigma}_{2,n}(\tilde{\delta}_n + \Delta_n) - \sigma_{1,n}(\tilde{\delta}_n) - (\widehat{\sigma}_{2,n}(\tilde{\delta}_n + \Delta_n) - \sigma_{1,n}(\tilde{\delta}_n) )\big\}
\\
+ \frac{\sigma_{1,n}\tilde{\delta}_n + \Delta_n) - \sigma_{1,n}(\tilde{\delta}_n)}{\Delta_n} - \sigma_{1,n}'(\tilde{\delta}_n) + \sigma_{1,n}'(\tilde{\delta}_n) - \sigma_\infty'(\tilde{\delta}_n). \label{main_diff-lemma_sigma_prime_hat}
\end{gather}

We have that
\begin{gather}
\widehat{\sigma}_{2,n}(\tilde{\delta}_n) - \sigma_{1,n}(\tilde{\delta}_n) = \frac{\widehat{\sigma}^2_{2,n}(\tilde{\delta}_n) - \sigma_{1,n}^2(\tilde{\delta}_n)}{\widehat{\sigma}_{2,n}(\tilde{\delta}_n) + \sigma_{1,n}(\tilde{\delta}_n)}.
\end{gather}

For $\tilde{\delta}_n$ slow enough, we have that $\widehat{\sigma}_{2,n}(\tilde{\delta}_n) \sim_P \sigma_\infty(\tilde{\delta}_n)$ and $\sigma_{1,n}(\tilde{\delta}_n) \sim_P \sigma_\infty(\tilde{\delta}_n)$. Therefore, recalling from the proof of lemma \ref{lemma-diff_sigma_hat-sigma_inf} that $\widehat{\sigma}^2_{2,n}(\tilde{\delta}_n) - \sigma_{1,n}^2(\tilde{\delta}_n) = O_P (\tilde{\delta}_n^{-k_4} (l_{2,n} - l_{1,n})^{-1/2})$, we have that
\begin{equation}
\widehat{\sigma}_{2,n}(\tilde{\delta}_n) - \sigma_{1,n}(\tilde{\delta}_n) \sim_P \frac{\widehat{\sigma}^2_{2,n}(\tilde{\delta}_n) - \sigma_\infty^2(\tilde{\delta}_n)}{2\sigma_\infty(\tilde{\delta}_n)} = O_P(\tilde{\delta}_n^{\gamma_{0, \infty} - k_2} (l_{2,n} - l_{1,n})^{-1/2}).\label{finite_diff_emp_process-lemma_sigma_prime_hat}
\end{equation}
Similarly, 
\begin{gather}
\widehat{\sigma}_{2,n}(\tilde{\delta}_n + \Delta_n) - \sigma_{1,n}(\tilde{\delta}_n + \Delta_n) \sim_P \frac{\widehat{\sigma}^2_{2,n}(\tilde{\delta}_n + \Delta_n) - \sigma_\infty^2(\tilde{\delta}_n + \Delta_n)}{2\sigma_\infty(\tilde{\delta}_n + \Delta_n)}\\
 = O_P(\tilde{\delta}_n^{\gamma_{0, \infty} - k_2} (l_{2,n} - l_{1,n})^{-1/2}).\label{finite_diff_emp_process-lemma_sigma_prime_hat-bis}
\end{gather}

Besides, note that using assumption \textbf{A2} and then assumption \textbf{A1}, we have that
\begin{gather}
\frac{\sigma_{1,n}(\tilde{\delta}_n + \Delta_n) - \sigma_{1,n}(\tilde{\delta}_n)}{\Delta_n} - \sigma_{1,n}'(\tilde{\delta}_n) = O_P(\tilde{\delta}_n^{-k_1} \Delta_n).\label{diff_fin-diff-sigma_prime-lemma_sigma_prime_hat}
\end{gather}

Finally, recall that from \textbf{A2} we have $\sigma_{1,n}'(\tilde{\delta}_n) - \sigma'_\infty(\tilde{\delta}_n) = O_P (\tilde{\delta}^{-k'_2} (l_{2,n} - l_{1,n})^{-\kappa'_2})$.
Injecting this latter identity,  \eqref{finite_diff_emp_process-lemma_sigma_prime_hat}, \eqref{finite_diff_emp_process-lemma_sigma_prime_hat-bis} and \eqref{diff_fin-diff-sigma_prime-lemma_sigma_prime_hat} into \eqref{main_diff-lemma_sigma_prime_hat}, we readily obtain that
\begin{gather}
\widehat{\sigma}_{2,n}(\tilde{\delta}_n) - \sigma_\infty(\tilde{\delta}_n) = O_P(\tilde{\delta}_n^{-k_5} \left(l_{2,n} - l_{1,n}\right)^{-1/4}),
\end{gather}
with $k_6 \equiv \min(k_1, k_2 - \gamma_{0, \infty}, k'_2)$ and $\kappa_6 \equiv \min(\kappa'_2, 1/4)$.
\end{proof}

\medskip

\begin{proof}[Proof of lemma \ref{lemma-rate-estimators}]
We prove the claims $\widehat{\beta}_n - \beta_0 = o_P (1 / \log n)$ and $\widehat{C}_{b', n} \xrightarrow{P} C_{b',0}$. The proofs of the remaining claims are identical.

\medskip

Observe that lemma \ref{lemma-diff_b_prime_hat-b_prime} holds under assumptions \textbf{A1} through \textbf{A5}. Let $i \in \{1,2\}$.

Therefore for $\tilde{\delta}_{i,n}$ slow enough,
\begin{gather}
\log \widehat{b'}_{i,n}(\tilde{\delta}_{i,n}) = \log b'_0(\tilde{\delta}_{i,n}) + \log (1 + O_P (\tilde{\delta}_{i,n}^{-k_4}(l_{2,n} - l_{1,n})^{-\kappa_4})).
\end{gather}

Therefore for $\tilde{\delta}_{i,n}$ slow enough, taking a first order Taylor expansion, for some $k_7 > 0$,
\begin{gather}
\log \widehat{b'}_{i,n}(\tilde{\delta}_{i,n}) = \log b'_0(\tilde{\delta}_{i,n}) + O_P (\tilde{\delta}_{i,n}^{-k_7})\label{proof_rates_estimators-Taylor_expansion_of_log}
\end{gather}

Recall that under \textbf{A1} $b_0'(\tilde{\delta}_{i,n}) \sim C_{b',0}\tilde{\delta}_{i,n}^{\beta_0 - 1}$. Thus $\log b'_0(\tilde{\delta}_n) = \log C_{b',0} + (\beta_0 - 1) \log \tilde{\delta}_n + o_P (\log \tilde{\delta}_{i,n} ) = \log C_{b',0} + (\beta_0 - 1) \log \tilde{\delta}_n + o_P (1 / \log n )$, where the last equality follows from the fact that $\tilde{\delta}_{i, n} < n^{-r^-}$ for some $r^- > 0$. Therefore, injecting this and \eqref{proof_rates_estimators-Taylor_expansion_of_log} into the definition of $\widehat{\beta}_n$, we obtain
\begin{gather}
\widehat{\beta}_n - \beta_0 = o_P\left(\frac{1}{\log n}\right).
\end{gather}

Then, we have that
\begin{gather}
\widehat{C}_{b',n} \equiv \widehat{b'}_{2,n}(\tilde{\delta}_{3,n}) \tilde{\delta}_{3,n}^{-(\widehat{\beta}_n - 1)} \sim_P C_{b',0} \tilde{\delta}_{3,n}^{\beta_0 - 1 -(\widehat{\beta}_n - 1)}\\
\sim_P C_{b',0} \tilde{\delta}_{3,n}^{\frac{1}{\log n}} \sim_P C_{b',0} \tilde{\delta}_{3,n}^\frac{1}{\log \tilde{\delta}_{3,n}} \xrightarrow{P} C_{b',0},
\end{gather}

where we have used the fact that $o_P(1/ \log n) = o_P (1 / \log \tilde{\delta}_{3,n})$ since $n^{-r^+} < \tilde{\delta}_{3,n} < n^{-r^-}$ for some $r^+ > 0, r^- > 0$.

\end{proof}

\medskip

\subsection{Asymptotic analysis of the single-split cross-validated one-step estimator}

\subsubsection{Technical lemmas for the asymptotic normality of the single-split cross-validated one-step estimator}

\begin{lemma}\label{link_gamma_nu-lemma}
Assume \textbf{A1}. Then $b_0(\delta) \sim_P C_{b',0}\beta_0^{-1} \delta^{\beta_0}$.
If $\sigma_\infty(\delta) \xrightarrow{\delta \rightarrow 0} \infty$, then $\gamma = -\nu - 1$. If $\sigma_\infty(\delta)$ has a finite limit as $\delta$ converges to zero, then $\gamma_{0, \infty} = 0$ and $\nu_{0, \infty} > -1$.
\end{lemma}

\begin{proof}
Integrating $b'_0(\delta) = C_{b', 0} \delta^{-\beta_0 - 1} + o(\delta^{-\beta_0 - 1})$ and noting that $b_0(0) = 0$ gives that $b_0(\delta) \sim_P C_{b',0}\beta_0^{-1}\delta^{\beta_0}$.

\medskip

Integrating $\sigma_\infty'(\delta) = C_{\sigma', \infty} \delta^{\nu_{0, \infty}} + o_P\left(\delta^{\nu_0}\right)$ yields $\sigma_\infty(\delta) = K +C_{\sigma', \infty}(\nu + 1)^{-1} \delta^{\nu_{0, \infty} + 1} + o_P(\delta^{\nu_{0, \infty} + 1})$, for some constant $K \in \mathbb{R}$.
If $\sigma_\infty(\delta) \xrightarrow{\delta \rightarrow 0} \infty$, then $K + C_{\sigma', \infty}(\nu + 1)^{-1} \delta^{\nu_{0, \infty} + 1} \sim \delta^{-\gamma_{0, \infty}}$ which implies that $\gamma_{0, \infty} = -\nu_{0, \infty} - 1 > 0$.

If $\sigma_\infty(\delta)$ converges to a positive finite constant as $\delta$ converges to zero, then $\sigma_\infty(\delta) \sim C_{\sigma, \infty} \delta^{-\gamma_0}$ implies that $\gamma = 0$. Besides, $\sigma_\infty(\delta) = K + C_{\sigma', \infty}(\nu + 1)^{-1} \delta^{\nu_{0, \infty} + 1} + o_P(\delta^{\nu_{0, \infty} + 1})$ implies that $\nu_{0,\infty} + 1 > 0$.
\end{proof}

\medskip

\subsubsection{Asymptotic optimality of the cross-validated single-split one-step estimator}

\begin{proof}

Observe that
\begin{align}
MSE_n(\delta) =& E_{P_0} \left[\left(\widehat{\Psi}_n(\delta) - \Psi(P_0) \right)^2 \right]\\
& E_{P_0} \left[ E_{P_0} \left[\left(\widehat{\Psi}_n(\delta) - \Psi(P_0) \right)^2 \bigg| S_{2,n} \right]\right].
\end{align}

We have that
\begin{gather}
E_{P_0} \left[\left(\widehat{\Psi}_n(\delta) - \Psi(P_0) \right)^2 \big| S_{2,n} \right] \\
= E_{P_0} \left[ \left( \left(P_{3,n} - P_0\right) D^*_\delta(\widehat{P}_{2,n})\right)^2 \big| S_{2,n}\right]
 + b_0^2(\delta) + R_\delta(\widehat{P}_{2,n}, P_0)^2 \\
+ 2 E_{P_0} \left[ \left(P_{3,n} - P_0\right) D^*_{\delta}(\widehat{P}_{2,n}) \big| S_{2,n} \right] \left(b_0(\delta) + R_\delta(\widehat{P}_{2,n}, P_0)\right)\\
+ 2 R_\delta(\widehat{P}_{2,n}, P_0) b_0(\delta)\\
= \frac{1}{l_{3,n} - l_{2,n}} \sigma^2_{2,n}(\delta) + b_0^2(\delta) + R_\delta(\widehat{P}_{2,n}, P_0)^2 + R_\delta(\widehat{P}_{2,n}, P_0) b_0(\delta). \label{conditional_MSE_developed_expression}
\end{gather}

Therefore
\begin{gather}
MSE_n'(\delta) = \frac{2}{l_{3,n} - l_{2,n}} \sigma_{2,n}(\delta) \sigma'_{2,n}(\delta) + 2 b_0(\delta) b'_0(\delta) \\
+ \frac{d R_\delta (\widehat{P}_{2,n}, P_0)}{d\delta} \left(b_0(\delta) + R_\delta (\widehat{P}_{2,n}, P_0) \right) + b'_0(\delta) R_\delta (\widehat{P}_{2,n}, P_0).
\end{gather}

Under assumptions \textbf{A7} and \textbf{A8}, evaluating the second line of the expression above at some $\delta_n$ with rate in a neighborhood of $r_{0, \infty}$, we have that
\begin{gather}
\frac{d R_\delta (\widehat{P}_{2,n}, P_0)}{d\delta}\big|_{\delta = \delta_n} \left(b_0(\delta) + R_{\delta_n} (\widehat{P}_{2,n}, P_0) \right) + b'_0(\delta_n) R_{\delta_n} (\widehat{P}_{2,n}, P_0)\\
= o_P \left( \left(\frac{1}{\sqrt{n}} \sigma_\infty'(\delta_n) + b_0'(\delta_n) \right) \left(\frac{\sigma_\infty(\delta_n)}{\sqrt{n}} + b_0(\delta_n) \right) \right)\\
= o_P \left(\frac{1}{n} \delta_n^{-\gamma + \nu} + \frac{1}{\sqrt{n}} \delta_n^{-\gamma + \beta - 1} + \frac{1}{\sqrt{n}} \delta_n^{\beta + \nu} + \delta_n^{2 \beta - 1}\right).
\end{gather}

$\bullet$ If $\sigma_\infty(\delta) \xrightarrow{\delta \rightarrow 0} \infty$, then, by lemma \ref{link_gamma_nu-lemma}, $\nu = -\gamma -1$. We then have that 
\begin{gather}
\frac{1}{n} \delta_n^{-\gamma + \nu} + \frac{1}{\sqrt{n}} \delta_n^{-\gamma + \beta - 1} + \frac{1}{\sqrt{n}} \delta_n^{\beta + \nu} + \delta_n^{2 \beta - 1} \\
= \frac{1}{n} \delta_n^{-2\gamma - 1} + 2\frac{1}{\sqrt{n}} \delta_n^{-\gamma + \beta -1} + \delta_n^{2 \beta - 1} \\
\lesssim \left( \frac{1}{\sqrt{n}} \sqrt{\sigma_\infty(\delta_n) \sigma_\infty'(\delta_n)} + \sqrt{b_0(\delta_n) b'_0(\delta_n)} \right)^2 \\
\leq \frac{2}{n}\sigma_\infty(\delta_n) \sigma_\infty'(\delta_n) +  2b_0(\delta_n) b'_0(\delta_n).
\end{gather}

$\bullet$ If $\sigma_\infty(\delta)$ has a finite limit as $\delta$ converges to zero, then, by lemma \ref{link_gamma_nu-lemma}, $\gamma = 0$ and $\nu > - 1$. We then have that
\begin{gather}
\frac{1}{n} \delta_n^{-\gamma + \nu} + \frac{1}{\sqrt{n}} \delta_n^{-\gamma + \beta - 1} + \frac{1}{\sqrt{n}} \delta_n^{\beta + \nu} + \delta_n^{2 \beta - 1} \\
= \frac{1}{n} \delta_n^{\nu} + \frac{1}{\sqrt{n}} \delta_n^{\beta - 1} + \frac{1}{\sqrt{n}} \delta_n^{\beta + \nu} + \delta_n^{2 \beta - 1}\\
\leq \frac{1}{n} \delta_n^{\nu} + 2\frac{1}{\sqrt{n}} \delta_n^{\beta - 1} + \delta_n^{2 \beta - 1} \lesssim \left( \frac{1}{\sqrt{n}} \sqrt{\sigma_\infty(\delta_n) \sigma_\infty'(\delta_n)} + \sqrt{b_0(\delta_n) b'_0(\delta_n)} \right)^2\\
\leq \frac{2}{n}\sigma_\infty(\delta_n) \sigma_\infty'(\delta_n) +  2b_0(\delta_n) b'_0(\delta_n).
\end{gather}

Therefore, for $\delta_n$ with rate in a neighborhood of $r_{0, \infty}$, we have that
\begin{gather}
\frac{d}{d\delta}  \left(E_{P_0}\left[ \left(\widehat{\Psi}_n(\delta) - \Psi(P_0) \right)^2 \big| S_{2,n} \right]\right)\bigg|_{\delta = \delta_n} = \frac{2}{l_{3,n} - l_{2,n}}\sigma_\infty(\delta_n) \sigma_\infty'(\delta_n) \\
+  2b_0(\delta_n) b'_0(\delta_n)  + o_P \left(\frac{2}{l_{3,n} - l_{2,n}}\sigma_\infty(\delta_n) \sigma_\infty'(\delta_n) +  2b_0(\delta_n) b'_0(\delta_n) \right) \\
= f(\delta_n) + o_P \left(f(\delta_n)\right),
\end{gather}

with $f(\delta_n) \equiv \frac{2}{l_{3,n} - l_{2,n}}\sigma_\infty(\delta_n) \sigma_\infty'(\delta_n) +  2b_0(\delta_n) b'_0(\delta_n)$.

Otherwise stated, for $\delta_n$ with rate in a neighborhood of $r_{0, \infty}$,
\begin{gather}
f(\delta_n)^{-1} \frac{d}{d\delta}  \left( E_{P_0} \left[ \left(\widehat{\Psi}_n(\delta_n) - \Psi(P_0) \right)^2 \big| S_{2,n} \right] \right) \big|_{\delta = \delta_n} \xrightarrow{P} 1.
\end{gather}

Therefore, using dominated convergence, we readily obtain that
\begin{gather}
MSE'_n(\delta_n) \sim_P \frac{2}{l_{3,n} - l_{2,n}}\sigma_\infty(\delta_n) \sigma_\infty'(\delta_n) +  2b_0(\delta_n) b'_0(\delta_n)\\
\sim_P \frac{2}{l_{3,n} - l_{2,n}} C_{\sigma, \infty} C_{\sigma', \infty} \delta_n^{-\gamma + \nu} + \frac{C_{b',0}^2}{\beta_0}\delta_n^{2 \beta - 1}.
\end{gather}

Therefore, 
\begin{gather}
\delta_n^* \sim_P \left(\frac{C_{\sigma,\infty} C_{\sigma', \infty} \beta}{C_{b',0}^2}
\right)^\frac{1}{2\beta - 1 + \gamma - \nu} \left(l_{3,n} - l_{2,n}\right)^{-\frac{1}{2\beta - 1 + \gamma - \nu}}.
\end{gather}

Under assumption \textbf{A7}, using expression \eqref{conditional_MSE_developed_expression} and dominated convergence, we readily obtain that, for $\delta_n$ with a rate in a neighborhood of $r_{0, \infty}$,
\begin{gather}
MSE_n(\delta_n) \sim_P \frac{1}{l_{3,n} - l_{2,n}} \sigma_\infty(\delta_n)^2 + b_0(\delta_n)^2. \label{MSE_n_equivalent}
\end{gather}

Recalling that from lemma \ref{lemma-rate-estimators} $\widehat{\delta}_n/\delta_{0,n}^* \xrightarrow{P} 1$, and injecting the expression of $\widehat{\delta}_{\epsilon, n}$ into \eqref{MSE_n_equivalent} yields the claim.
\end{proof}

\medskip

\subsection{Asymptotic normality of the cross-validated Targeted Maximum Likelihood estimator}

The proof of theorem \ref{asymptotic_normality-CV-TMLE} relies on lemma 2 from \cite{zheng_vdl2010}, which is an equicontinuity result for a certain class of functions. We reproduce  it here for our reader's convenience. We first recall the definition of the entropy of a class of functions $\mathcal{G}$:
\begin{gather}
Entro(\mathcal{G}) \equiv \int_0^\infty \sqrt{\log \sup_\Lambda N\left(\epsilon \left\lVert G \right\rVert_{\Lambda, 2}, \mathcal{G}, L_2(\Lambda) \right)} d\epsilon,
\end{gather}
where $G$ is the envelope of $\mathcal{G}$.

\medskip

\begin{lemma}\label{lemma-CV-TMLE}
Suppose $| \epsilon_n - \epsilon_0 | \xrightarrow{P} 0$ for some $\epsilon_0 \in \mathbb{R}$. For each sample split $B_n$, consider a class of measurable functions of $O$ 
\begin{gather}
\mathcal{G}\left(P^0_{n, B_n}\right) \equiv \left\lbrace g_\epsilon \left(P^0_{n, B_n}\right) \equiv g\left(\epsilon, P^0_{n, B_n}\right) - g(\epsilon_0, P_0) : \epsilon \right\rbrace,
\end{gather}
where the index set contains $\epsilon_n$ with probability tending to one. For a deterministic sequence $\delta_n \rightarrow 0$, define the subclasses 
\begin{gather}
\mathcal{G}_{\delta_n} \left(P^0_{n, B_n}\right) \equiv \left\lbrace g_\epsilon \left(P^0_{n, B_n}\right) \equiv g\left(\epsilon, P^0_{n, B_n}\right) - g(\epsilon_0, P_0) : |\epsilon_n - \epsilon_0| \leq \delta_n \right\rbrace.
\end{gather}

If for determinitic sequences $\delta_n \rightarrow 0$ we have 
\begin{gather}
E \left\lbrace Entro\left(\mathcal{G}_{\delta_n} \left(P^0_{n, B_n}\right)\right) \sqrt{P_0 G\left(\delta_n, P^0_{n, B_n}\right)^2} \right\rbrace \rightarrow 0 \text{ as } n \rightarrow 0,
\end{gather}

where $G\left(\delta_n, P^0_{n, B_n}\right)$ is the envelope of $\mathcal{G}_{\delta_n} \left(P^0_{n, B_n}\right)$, then 
\begin{gather}
\sqrt{n} \left(P^1_{n, B_n} - P_0 \right) \left\lbrace g\left(\epsilon, P^0_{n, B_n}\right) - g(\epsilon_0, P_0) \right\rbrace = o_P(1).
\end{gather}
\end{lemma}

\medskip

\begin{proof}[Proof of theorem \ref{asymptotic_normality-CV-TMLE}]
Notice that, by lemma \ref{lemma-rate-estimators}, we have that $\hat{\delta}_n/\delta_{0, \infty, n} \xrightarrow{p} 1$ and $\hat{\delta}_n^{\hat{\gamma}_n}/\delta_{0, \infty, n}^{\gamma_{0, \infty}} \xrightarrow{p} 1$. 

Observe that
\begin{gather}
\widehat{\Psi}_n^{CV-TMLE} - \Psi(P_0) = E_{B_n} \Psi(\widehat{P}^*_{n, B_n, \hat{\delta}_n}) - \Psi_{\hat{\delta}_n}(P_0) \\
+ \Psi_{\hat{\delta}_n}(P_0) - \Psi(P_0) \\
= E_{B_n} \left(P^1_{n, B_n} - P_0\right) D^*_{\hat{\delta}_n} (\widehat{P}^*_{n, B_n, \hat{\delta}_n}) - \Psi_{\hat{\delta}_n}(P_0) \label{empirical_process_CV-TMLE}\\
+ E_{B_n} R_{\hat{\delta}_n} (\widehat{P}^*_{n, B_n, \hat{\delta}_n}, P_0) \label{remainder_CV-TMLE} \\
+ \Psi_{\hat{\delta}_n}(P_0) - \Psi(P_0). \label{bias_CV-TMLE}
\end{gather}

$\bullet$ Analysis of the bias term \eqref{bias_CV-TMLE} is exactly the same as in the proof of theorem \ref{asymptotic_normality_thm}. We thus have that
\begin{gather}
\widehat{C}_{\sigma, n}^{-1} \sqrt{n} \hat{\delta}_n^{\hat{\gamma}_n} \left(\Psi_{\hat{\delta}_n}(P_0) - \Psi(P_0) \right) = o_P(1).
\end{gather}

$\bullet$ Using the same arguments as in the proof of theorem \ref{asymptotic_normality_thm} proves that $\widehat{C}_{\sigma, n}^{-1} \hat{\delta}_n^{\hat{\gamma}_n} R_{\hat{\delta}_n}(\widehat{P}^*_{n, B_n, \hat{\delta}_n}, P_0) = o_P(1)$. Since $B_n$ ranges over a finite set, we have that
\begin{gather}
\widehat{C}_{\sigma, n}^{-1} \hat{\delta}_n^{\hat{\gamma}_n} E_{B_n}R_{\hat{\delta}_n}(\widehat{P}^*_{n, B_n, \hat{\delta}_n}, P_0) = o_P(1).
\end{gather}

$\bullet$ Let us now turn to the analysis of the empirical process term \eqref{empirical_process_CV-TMLE}. We have that
\begin{gather}
E_{B_n} \left(P^1_{n, B_n} - P_0\right) D^*_{\hat{\delta}_n} (\widehat{P}^*_{n, B_n, \hat{\delta}_n}) - \Psi_{\hat{\delta}_n}(P_0) \\
= E_{B_n} \left(P^1_{n, B_n} - P_0\right) D^*_{\hat{\delta}_n} \left(P_\infty\right) \label{emp_proc-CV-TMLE-I} \\
+ E_{B_n} \left(P^1_{n, B_n} - P_0\right) \left(D^*_{\hat{\delta}_n} (\widehat{P}^*_{n, B_n, \hat{\delta}_n}) - D^*\left(P_\infty\right) \right). \label{emp_proc-CV-TMLE-II}
\end{gather}

Let us characterize term \eqref{emp_proc-CV-TMLE-II}. Remember that $\widehat{P}^*_{n, B_n, \hat{\delta}_n} = \widehat{P}_{n, B_n, \widehat{\delta}_n, \epsilon_n}$ and that $\widehat{P}_{n, B_n, \delta, \epsilon}$ depends on the sample only through $\widehat{P}^0_{n, B_n}$. Thus, by application of lemma \ref{lemma-CV-TMLE}, and using that 
$\frac{\hat{\delta}_n^{\hat{\gamma}_n}}{\delta_{0, \infty, n}^{\gamma_{0, \infty}}} \xrightarrow{p} 1$, we obtain
\begin{gather}
\sqrt{n} \hat{\delta}_n^{\hat{\gamma}_n} E_{B_n} \left(P^1_{n, B_n} - P_0\right) \left(D^*_{\hat{\delta}_n} (\widehat{P}^*_{n, B_n, \hat{\delta}_n}) - D^*\left(P_\infty\right) \right) = o_P(1).
\end{gather}

\medskip

We now analyze term \eqref{emp_proc-CV-TMLE-I}. By the Lindeberg central limit theorem for triangular arrays, we have that $C_{\sigma, 0, \infty}^{-1} \sqrt{n} \delta_{0, \infty, n}^{\gamma_{0, \infty}} E_{B_n} (P^1_{B_n, n} - P_0 ) D^*_{\delta_{0, \infty, n}}(P_\infty) \xrightarrow{d} \mathcal{N}(0,1).$
Besides,
\begin{gather}
\sqrt{n} \widehat{\delta}_n^{\hat{\gamma}_n} E_{B_n} \left(P^1_{n, B_n} - P_0 \right) D^*_{\hat{\delta}_n}(P_\infty) \\
- \sqrt{n} \delta_{0, \infty, n}^{\gamma_{0, \infty}} E_{B_n} \left(P^1_{n, B_n} - P_0 \right) D^*_{\delta_{0, \infty, n}}(P_\infty)\\
= \sqrt{n} \left(\hat{\delta}_n^{\hat{\gamma}_n} - \delta_{0, \infty, n}^{\gamma_{0, \infty}} \right) D^*_{\delta_{0, \infty, n}}(P_\infty) \label{emp_proc-CV-TMLE-III}\\
+ \sqrt{n} \hat{\delta}_n^{\hat{\gamma}_n} E_{B_n} \left(P^1_{n, B_n} - P_0 \right)  \left(D^*_{\hat{\delta}_n}(P_\infty) - D^*_{\delta_{0, \infty, n}}(P_\infty) \right).\label{emp_proc-CV-TMLE-IV}
\end{gather}

Since $\hat{\delta}_n^{\hat{\gamma}_n} - \delta_{0, \infty, n}^{\gamma_{0, \infty}} = o_P \left(\delta_{0, \infty, n}^{\gamma_{0, \infty}}\right)$, we have that \eqref{emp_proc-CV-TMLE-III} is $o_P(1)$.

Besides, using assumption \textbf{A11}, we have that $\delta_{0, \infty, n} D^*_{\hat{\delta}_n}(P_\infty) - L_\infty = \delta_{0, \infty, n}/\hat{\delta}_n (\hat{\delta}_n D^*_{\hat{\delta}_n}(P_\infty) - L_\infty ) + (\delta_{0, \infty, n}/\hat{\delta}_n - 1) L_\infty = o_P(1)$. Thus
\begin{gather}
\left\lVert \delta_{0, \infty, n} \left( D^*_{\hat{\delta}_n}(P_\infty) - D^*_{\delta_{0, \infty, n}}(P_\infty) \right) \right\rVert_{L_2(P_0)}\\
\leq \left\lVert \delta_{0, \infty, n}  D^*_{\hat{\delta}_n}(P_\infty) - L_\infty \right\rVert_{L_2(P_0)}+ \left\lVert L_\infty - \delta_{0, \infty, n} D^*_{\delta_{0, \infty, n}}(P_\infty) \right\rVert_{L_2(P_0)} \\
= o_P(1).
\end{gather}

Therefore, by lemma \ref{lemma-CV-TMLE}, \eqref{emp_proc-CV-TMLE-IV} is $o_P(1)$. 
\end{proof}

\medskip

\subsection{Verification of the examples' hypothesis}

The following lemma states a useful inequality in the context of example 2 (mean counterfactual outcome $EY_d$.)

\begin{lemma}\label{lemma_key_inequality_example2}
Consider the setting of example 2. We have that
\begin{gather}
|\sigma_n(\delta) - \sigma_\infty(\delta)|^2 \leq P_0 \frac{g_0(1|W)}{g_{0, \delta}^2(1|W)} \left(\hat{\bar{Q}}_n - \bar{Q}_\infty \right)^2.
\end{gather}

\end{lemma}

\begin{proof}
Observe that
\begin{gather}
\left|\sigma_n^2(\delta) - \sigma^2_\infty(\delta) \right| = \left| P_0 \left( D^*_\delta (\widehat{P}_n ) - P_0 D^*_\delta (\widehat{P}_n ) \right)^2  - P_0 \left( D^*_{\delta, \infty} - P_0 D^*_{\delta, \infty} \right)^2 \right|\\
= P_0 \bigg[ \left\lbrace\left( D^*_\delta (\widehat{P}_n ) - P_0 D^*_\delta (\widehat{P}_n ) \right)  - \left( D^*_{\delta, \infty} - P_0 D^*_{\delta, \infty} \right) \right\rbrace \\
\times \left\lbrace\left( D^*_\delta (\widehat{P}_n ) - P_0 D^*_\delta (\widehat{P}_n ) \right)  + \left( D^*_{\delta, \infty} - P_0 D^*_{\delta, \infty} \right)  \right\rbrace \bigg]\\
\leq  \sqrt{P_0 \left(\Delta_{\delta, n, \infty} - P_0 \Delta_{\delta, n, \infty} \right)^2} \left(\sigma_n(\delta) + \sigma_\infty(\delta) \right),
\end{gather}

where  
\begin{equation}
\Delta_{\delta, n, \infty} \equiv D^*_\delta\left( \widehat{P}_n \right) - D^*_{\delta, \infty},
\end{equation}
and the last inequality comes from Cauchy-Schwarz.

Therefore,
\begin{equation}
\left| \sigma_n(\delta) - \sigma_\infty(\delta) \right| \leq \sqrt{P_0 \left(\Delta_{\delta, n, \infty} - P_0 \Delta_{\delta, n, \infty} \right)^2}.
\end{equation}

For any $P$, we have that 
\begin{gather}\label{DminusP0D}
D^*_\delta(P) - P_0 D^*_\delta(P) = \frac{A}{g_{0, \delta}(A|W)}\left(\bar{Q}_0 - \bar{Q}\right) \\
+ \frac{g_0(1|W)}{g_{0, \delta}(1|W)} \bar{Q} - P_0 \frac{g_0(1|W)}{g_{0, \delta}(1|W)} \bar{Q}.
\end{gather}

Thus 
\begin{gather}
\Delta_{\delta, n, \infty} - P_0 \Delta_{\delta, n, \infty} = \frac{A}{g_{0, \delta}(A|W)}\left(\bar{Q}_\infty - \hat{\bar{Q}}_n\right) + \frac{g_0(1|W)}{g_{0, \delta}(1|W)} \left(\hat{\bar{Q}}_n - \bar{Q}_\infty \right).
\end{gather}

Therefore
\begin{gather}
P_0 \left(\Delta_{\delta, n, \infty} - P_0 \Delta_{\delta, n, \infty} \right)^2 = P_0 \frac{g_0(1|W)}{g_{0, \delta}^2(1|W)} \left(\hat{\bar{Q}}_n - \bar{Q}_\infty \right)^2 \\
+ P_0 \frac{g_0^2(1|W)}{g_{0, \delta}^2(1|W)} \left(\hat{\bar{Q}}_n - \bar{Q}_\infty \right)^2 - 2 P_0 \frac{g_0^2(1|W)}{g_{0, \delta}^2(1|W)} \left(\hat{\bar{Q}}_n - \bar{Q}_\infty \right)^2\\
\leq P_0 \frac{g_0(1|W)}{g_{0, \delta}^2(1|W)} \left(\hat{\bar{Q}}_n - \bar{Q}_\infty \right)^2.
\end{gather}

\end{proof}

\begin{proof}[Proof of lemma \ref{lemma-preA2-example2}]
Observe that $g_0(1|W)/g_{0, \delta}^2(1|W) \leq \delta^{-1}$. The result then directly follows by injecting this in the inequality from lemma \ref{lemma_key_inequality_example2}.
\end{proof}

\begin{proof}[Proof of lemma \ref{lemma-sigma_n_sigma_inf-EY1-example}, mean counterfactual outcome case]
From lemma \ref{lemma_key_inequality_example2}, we have that
\begin{gather}
|\sigma_n(\delta) - \sigma_\infty(\delta) |^2 \leq \left(P_0 \frac{g_0(1|W)}{g_{0, \delta}^2(1|W)} \left(\bar{Q}_0 (1 - \bar{Q}_0) + \left(\bar{Q}_\infty - \bar{Q}_0\right)^2 \right) \right) \\
\times 
\bigg\| \frac{\left(\hat{\bar{Q}}_n - \bar{Q}_\infty \right)^2}{\bar{Q}_0 (1 - \bar{Q}_0) + \left(\bar{Q}_\infty - \bar{Q}_0\right)^2} \bigg\|_\infty. \label{bound_on_DeltaVariance}
\end{gather}

Besides, from \eqref{DminusP0D}, we have that
\begin{gather}
\sigma_\infty^2(\delta) = P_0 \frac{g_0(1|W)}{g_{0, \delta}(1|W)^2} \left(\bar{Q}_0 (1 - \bar{Q}_0) + \left(\bar{Q_0} - \bar{Q}_\infty\right)^2 \right) \\
+ P_0 \left(\frac{g_0(1|W)}{g_{0, \delta}(1|W)}\bar{Q}_\infty - P_0 \frac{g_0(1|W)}{g_{0, \delta}(1|W)}\bar{Q}_\infty \right)^2 \\
+2 P_0 \left\lbrace\frac{g_0(1|W)}{g_{0, \delta}(1|W)} \left(\bar{Q}_0 - \bar{Q}_\infty \right) \left(\frac{g_0(1|W)}{g_{0, \delta}(1|W)} \bar{Q} -  P_0 \left(\frac{g_0(1|W)}{g_{0, \delta}(1|W)} \bar{Q}_0 \right) \right) \right\rbrace. \label{sigma_infty_expression}
\end{gather}

Denote $I^2_\infty(\delta) \equiv P_0 \frac{g_0(1|W)}{g_{0, \delta}(1|W)^2} \left(\bar{Q}_0 (1 - \bar{Q}_0) + \left(\bar{Q_0} - \bar{Q}_\infty\right)^2 \right)$. 

\medskip

Let $\delta_n$ be a non-negative sequence that converges to zero. The inequality \eqref{bound_on_DeltaVariance} shows that $|\sigma_n(\delta_n) - \sigma_\infty(\delta_n)| = o_P\left(I_\infty(\delta_n)\right)$.

\medskip

We now distinguish two situations.

$\bullet$ If $\sigma_\infty(\delta) \xrightarrow{\delta \rightarrow 0} \infty$ then $\sigma_\infty^2(\delta) \sim I^2_\infty(\delta)$, since the remaining terms in expression \eqref{sigma_infty_expression} are bounded. Therefore, we have $\sigma_n(\delta_n) - \sigma_\infty(\delta_n) = o_P \left(\sigma_\infty(\delta_n)\right)$, which is the desired result.

$\bullet$ If $\sigma_\infty(\delta) \xrightarrow{\delta \rightarrow 0} C < \infty$, then $I^2_\infty(\delta)$ is bounded as a difference of bounded terms (as can be seen from expression \eqref{sigma_infty_expression}. Therefore $\sigma_n(\delta_n) - \sigma_\infty(\delta_n) = o_P(1)$, and since $C > 0$, this implies $\sigma_n(\delta_n) - \sigma_\infty(\delta_n) = o_P \left(\sigma_\infty(\delta_n)\right)$.

\end{proof}

The following lemma provides a useful inequality in the context of the dose response curve example.

\begin{lemma}\label{lemma_key_inequality_example3}
Consider the setting of example 3. We have that
\begin{gather}
|\sigma_n(\delta) - \sigma_\infty(\delta)|^2 \leq P_0 \int \frac{K^2_{a_0, \delta}(a)}{g_0(a|W)} \left(\hat{\bar{Q}}_n - \bar{Q}_\infty \right)^2.
\end{gather}
\end{lemma}

\begin{proof}
Using the same notations and following the same derivation as in the proof of lemma \ref{lemma_key_inequality_example2}, we have that
\begin{gather}
\left| \sigma_n(\delta) - \sigma_\infty(\delta) \right| \leq \sqrt{P_0 \left(\Delta_{\delta, n, \infty} - P_0 \Delta_{\delta, n, \infty} \right)^2}.
\end{gather}

We will drop the $a_0$ subscript for notational convenience.

\medskip 

For all $P$, we have that
\begin{gather}\label{DminusP0_DRC}
D^*_\delta(P) - P_0 D^*_\delta(P) = \frac{K_{\delta, a_0}(A)}{g_0(A|W)} \left(Y - \bar{Q}\right) \\
+ \int_a K_{a_0, \delta	}(a) \bar{Q}(a, W) da - \Psi_\delta(P_0).
\end{gather}

Thus, 
\begin{gather}
\Delta_{\delta, n, \infty} - P_0 \Delta_{\delta, n, \infty} = \frac{K_{a_0, \delta}(A)}{g_0(A|W)} \left(\bar{Q}_\infty - \hat{\bar{Q}}_n \right) + \int_a K_{a_0, \delta}(a)\left(\hat{\bar{Q}}_n - \bar{Q}_\infty \right) da.
\end{gather}

Therefore,

\begin{gather}
P_0 \left(\Delta_{\delta, n, \infty} - P_0 \Delta_{\delta, n, \infty} \right)^2 = P_0 \int_a \frac{K_{a_0, \delta}^2(a)}{g_0(a|W)}\left(\hat{\bar{Q}}_n - \bar{Q}_\infty \right)^2 da \\
+ P_0 \left(K_{a_0, \delta}(a)\left(\hat{\bar{Q}}_n - \bar{Q}_\infty \right) da\right)^2\\
- 2P_0 \left\lbrace \frac{K_{a_0, \delta}(A)}{g_0(A|W)}\left(\bar{Q}_\infty - \hat{\bar{Q}}_n \right) \int_a K_{a_0, \delta}(a)\left(\hat{\bar{Q}}_n - \bar{Q}_\infty \right) da\right\rbrace\\
= P_0 \int_a \frac{K_{a_0, \delta}^2(a)}{g_0(a|W)}\left(\hat{\bar{Q}}_n - \bar{Q}_\infty \right)^2 da - P_0 \left(K_{a_0, \delta}(a)\left(\hat{\bar{Q}}_n - \bar{Q}_\infty \right) da\right)^2\\
\leq P_0 \int_a \frac{K_{a_0, \delta}^2(a)}{g_0(a|W)}\left(\hat{\bar{Q}}_n - \bar{Q}_\infty \right)^2 da.
\end{gather}

\end{proof}

\begin{proof}[Proof of lemma \ref{lemma-preA2-example3}]
From lemma \ref{lemma_key_inequality_example3}, we have that
\begin{gather}
|\sigma_{2,n}(\delta) - \sigma_\infty(\delta)|^2 \leq P_0 \int \frac{K^2_{a_0, \delta}(a)}{g_0(a|W)} \left(\hat{\bar{Q}}_n - \bar{Q}_\infty \right)^2 da \\
\leq \left(\frac{1}{\delta^2} P_0\int K^2 \left(\frac{a - a_0}{\delta} \right) g_0^{-1}(a|W) da\right) \|\hat{\bar{Q}}_n - \bar{Q}_\infty\|_{L_\infty(P_0)}^2.
\end{gather}

A change of variables gives that
\begin{gather}
\frac{1}{\delta^2} P_0\int K^2\left(\frac{a-a_0}{\delta}\right)\left(g_0(a|W)\right)^{-1} da = \frac{1}{\delta} P_0 \int K^2(u) \left(g_0(a_0 + \delta u|W)\right)^{-1} du \\
\leq \delta^{-1}\|g_0^{-1}\|_\infty\int K^2(u) du.
\end{gather}
Hence the result.
\end{proof}

\begin{proof}[Proof of lemma \ref{lemma-sigma_n_sigma_inf-EY1-example}, dose-response curve example.]
From lemma \ref{lemma_key_inequality_example3}, we have that
\begin{gather}
|\sigma_{n}(\delta) - \sigma_\infty(\delta)|^2 \leq \left(P_0 \int_a \frac{K_{a_0, \delta}^2(a)}{g_0(a|W)}\left(\bar{Q}_0 (1 - \bar{Q}_0 ) + \left(\bar{Q}_\infty - \bar{Q}_0\right)^2 \right)da \right) \\
\times \left\lVert \frac{\left(\hat{\bar{Q}}_n - \bar{Q}_\infty\right)^2}{\bar{Q}_0 (1 - \bar{Q}_0 ) + \left(\bar{Q}_\infty - \bar{Q}_0 \right)^2} \right\rVert_{L_\infty(P_0)}. \label{sigma_diff_bound_lemma2DRC}
\end{gather}

Besides, from \eqref{DminusP0_DRC}, we have 
\begin{gather}
P_0 \left(D^*_\delta(P_\infty) - P_0 D^*_\delta(P_\infty) \right)^2 = P_0 \int_a \frac{K_{a_0, \delta}^2(a)}{g_0(a|W)}\left(\bar{Q}_0 (1 - \bar{Q}_0) + \left(\bar{Q}_0 - \bar{Q}_\infty\right)^2 \right) da \\
+ P_0 \left( \int_a K_{a_0, \delta}(a) \bar{Q}(a, W) da - \Psi_\delta(P_0) \right)^2\\
+ 2 P_0 \left\lbrace\int_a K_{a_0, \delta}(a) \bar{Q}(a, W) da \left(\int_a K_{a_0, \delta}(a) \bar{Q}(a, W) da - \Psi_\delta(P_0) \right) \right\rbrace.
\end{gather}

One readily shows that the first term is equivalent, as $\delta$ converges to zero, to $C \delta^{-1}$ for some $C > 0$, whereas the two remaining terms are bounded. Therefore, for any non negative sequence $\delta_n$ that converges to zero, we have that $\sigma_\infty^2(\delta) \sim I^2_\infty(\delta_n) \equiv P_0 \int_a \frac{K_{a_0, \delta_n}^2(a)}{g_0(a|W)}\left(\bar{Q}_0 (1 - \bar{Q}_0) + \left(\bar{Q}_0 - \bar{Q}_\infty\right)^2 \right) da$. Besides, from equation \eqref{sigma_diff_bound_lemma2DRC}, we have $\sigma_n(\delta_n) - \sigma_\infty(\delta_n) = o_P\left(I_\infty(\delta_)\right)$. Therefore, we have proved that $\sigma_n(\delta_n) - \sigma_\infty(\delta_n) = o_P \left(\sigma_\infty(\delta_n) \right)$, which is the wished result.
\end{proof}

\end{document}